\documentclass[10pt]{article}

\usepackage[colorlinks = true, pdfstartview = FitV, linkcolor = blue, citecolor = blue, urlcolor = blue, hypertexnames=false]{hyperref}
\usepackage{booktabs}
\usepackage{graphicx}
\usepackage{amsmath}
\usepackage{amssymb}
\usepackage{latexsym}
\usepackage{subfigure}
\usepackage{crop}
\usepackage{algorithmic,algorithm}
\usepackage{multirow}
\usepackage{bm}
\usepackage{bbm}
\usepackage{enumerate}
\usepackage{url}
\usepackage{array}
\usepackage{paralist}
\usepackage{diagbox}
\usepackage{booktabs}
\usepackage[dvipsnames]{xcolor}

\usepackage{rotating}
\usepackage[capitalise]{cleveref}
\crefname{equation}{}{}
\crefmultiformat{section}{\S\S#2#1#3}{ and~#2#1#3}{, #2#1#3}{, and~#2#1#3}
\crefname{figure}{Figure}{Figures}
\creflabelformat{equation}{\textup{(#2#1#3)}}
\crefname{assumption}{Assumption}{Assumptions}
\crefname{condition}{Condition}{Conditions}

\usepackage{xspace}

\renewcommand\th{\textsuperscript{th}\xspace}

\usepackage{fullpage}
\usepackage{multirow}

\usepackage[sort,nocompress]{cite}

\usepackage{arydshln}
\usepackage{enumitem}
\newlist{myenum}{enumerate}{3}
\setlist[myenum,1]{label=\textbf{\arabic*.},
	ref  =\textbf{\arabic*.}}
\setlist[myenum,2]{label=\textbf{(\alph*)},
	ref  =\themyenumi\textbf{(\alph*)}}
\setlist[myenum,3]{label=\bfseries(\roman*),
	ref  =\themyenumii\textbf{.(\roman*)}}

\crefname{myenumi}{item}{items}
\crefname{myenumii}{item}{items}
\crefname{myenumiii}{item}{items}

\usepackage{pifont}
\usepackage{accents}

\usepackage{stackengine}
\stackMath
\newcommand\tsup[2][2]{%
	\def\useanchorwidth{T}%
	\ifnum#1>1%
	\stackon[-.5pt]{\tsup[\numexpr#1-1\relax]{#2}}{\scriptscriptstyle\sim}%
	\else%
	\stackon[.5pt]{#2}{\scriptscriptstyle\sim}%
	\fi%
}

\definecolor{forestgreen}{rgb}{0.13, 0.55, 0.13}

\definecolor{amber}{rgb}{1.0, 0.75, 0.0}

\definecolor{bananayellow}{rgb}{.8, 0.6, 0}

\definecolor{cuhkpl}{RGB}{152,24,147}

\newcounter{comment}

\usepackage{amsthm}
\usepackage[framemethod=TikZ]{mdframed}

\newmdtheoremenv[%
linewidth = 1pt,%
roundcorner = 10pt,%
leftmargin = 0,%
rightmargin = 0,%
backgroundcolor = green!3,%
outerlinecolor = blue!70!black,%
splittopskip = \topskip,%
ntheorem = true,%
]{theorem}{Theorem}

\newmdtheoremenv[%
linewidth = 1pt,%
roundcorner = 10pt,%
leftmargin = 0,%
rightmargin = 0,%
backgroundcolor = green!3,%
outerlinecolor = blue!70!black,%
splittopskip = \topskip,%
ntheorem = true,%
]{corollary}{Corollary}

\newmdtheoremenv[%
linewidth = 1pt,%
roundcorner = 10pt,%
leftmargin = 0,%
rightmargin = 0,%
backgroundcolor = green!3,%
outerlinecolor = blue!70!black,%
splittopskip = \topskip,%
ntheorem = true,%
]{lemma}{Lemma}

\newmdtheoremenv[%
linewidth = 1pt,%
roundcorner = 10pt,%
leftmargin = 0,%
rightmargin = 0,%
backgroundcolor = green!3,%
outerlinecolor = blue!70!black,%
splittopskip = \topskip,%
ntheorem = true,%
]{proposition}{Proposition}

\newmdtheoremenv[%
linewidth = 1pt,%
roundcorner = 10pt,%
leftmargin = 0,%
rightmargin = 0,%
backgroundcolor = blue!3,%
outerlinecolor = blue!70!black,%
splittopskip = \topskip,%
ntheorem = true,%
]{definition}{Definition}

\newmdtheoremenv[%
linewidth = 1pt,%
roundcorner = 10pt,%
leftmargin = 0,%
rightmargin = 0,%
backgroundcolor = yellow!3,%
outerlinecolor = blue!70!black,%
splittopskip = \topskip,%
ntheorem = true,%
]{condition}{Condition}

\newmdtheoremenv[%
linewidth = 1pt,%
roundcorner = 10pt,%
leftmargin = 0,%
rightmargin = 0,%
backgroundcolor = yellow!3,%
outerlinecolor = blue!70!black,%
splittopskip = \topskip,%
ntheorem = true,%
]{assumption}{Assumption}

\theoremstyle{definition}
\newmdtheoremenv[%
linewidth = 1pt,%
roundcorner = 10pt,%
leftmargin = 0,%
rightmargin = 0,%
backgroundcolor = cyan!3,%
outerlinecolor = blue!70!black,%
splittopskip = \topskip,%
ntheorem = true,%
]{example}{Example}

\theoremstyle{definition}
\newmdtheoremenv[%
linewidth = 1pt,%
roundcorner = 10pt,%
leftmargin = 0,%
rightmargin = 0,%
backgroundcolor = red!3,%
outerlinecolor = blue!70!black,%
splittopskip = \topskip,%
ntheorem = true,%
]{remark}{Remark}

\mdtheorem[%
linewidth = 1pt,%
roundcorner = 10pt,%
leftmargin = 0,%
rightmargin = 0,%
backgroundcolor = green!3,%
outerlinecolor = blue!70!black,%
splittopskip = \topskip,%
ntheorem = true,%
]{informal}{Contributions and Main Results (Informal)}

\mdtheorem[%
linewidth = 1pt,%
roundcorner = 10pt,%
leftmargin = 0,%
rightmargin = 0,%
backgroundcolor = yellow!3,%
outerlinecolor = blue!70!black,%
splittopskip = \topskip,%
ntheorem = true,%
]{NPC}{NPC Condition}

\mdtheorem[%
linewidth = 1pt,%
roundcorner = 10pt,%
leftmargin = 0,%
rightmargin = 0,%
backgroundcolor = yellow!3,%
outerlinecolor = blue!70!black,%
splittopskip = \topskip,%
ntheorem = true,%
]{SOL}{Inexactness Condition}

\mdtheorem[%
linewidth = 1pt,%
roundcorner = 10pt,%
leftmargin = 0,%
rightmargin = 0,%
backgroundcolor = green!3,%
outerlinecolor = blue!70!black,%
splittopskip = \topskip,%
ntheorem = true,%
]{informal}{Main Result (Informal)}

	\usepackage{tikz}
	\usepackage{xparse}
	
	\NewDocumentCommand\DownArrow{O{2.0ex} O{black}}{%
		\mathrel{\tikz[baseline] \draw [<-, line width=0.5pt, #2] (0,0) -- ++(0,#1);}
	}
	
	\usepackage{listings} 
	
	\definecolor{mygreen}{rgb}{0,0.6,0}
	\definecolor{mygray}{rgb}{0.5,0.5,0.5}
	\definecolor{mymauve}{rgb}{0.58,0,0.82}
	
	\usepackage{dsfont}

	\newcommand{\real}{\mathbb{R}}

	\newcommand{\T}{\intercal}
	\DeclareMathOperator*{\argmin}{arg\,min}

	\newcommand {\dotprod}[1]{\left\langle #1\right\rangle}
	\newcommand {\vnorm}[1]{\left\| #1\right\|}
	\newcommand {\abs}[1]  {\left| #1 \right|}

	\newcommand {\rank}  {\textnormal{Rank}}
	
	\newcommand {\range}  {\textnormal{Range}}

	\newcommand {\Span}  {\textnormal{Span}}

	\newcommand {\zero}   {\mathbf{0}}
	\newcommand {\one}   {\mathbf{1}}
	
	\renewcommand {\AA}  {\mathbf{A}}

	\renewcommand {\aa}   {\mathbf{a}}
	
	\newcommand {\BB}  {\mathbf{B}}

	\newcommand {\bb}   {\mathbf{b}}

	\newcommand {\CC}   {\mathbf{C}}

	\newcommand {\DD}  {\mathbf{D}}

	\newcommand {\dd}   {\mathbf{d}}

	\newcommand {\ee}   {\mathbf{e}}

	\newcommand {\eye}  {\mathbf{I}}
	
	\newcommand {\pp}   {\mathbf{p}}

	\newcommand {\QQ}  {\mathbf{Q}}

	\newcommand {\RR}  {\mathbf{R}}

	\newcommand {\wRR}  {\widetilde{\RR}}

	\newcommand {\rr}   {\mathbf{r}}
	\newcommand {\rrk}  {{\rr}_{k}}

	\newcommand {\TT}  {\mathbf{T}}
	\newcommand {\TTt}  {\mathbf{T}_{t}}

	\newcommand {\wTT}  {\widetilde{\TT}}

	\newcommand {\ttt}  {\mathbf{t}}
	
	\newcommand {\VV}  {\mathbf{V}}

	\newcommand {\vv}   {\mathbf{v}}

	\newcommand {\ww}   {\mathbf{w}}

	\newcommand {\xx}   {\mathbf{x}}

	\newcommand {\xxk}  {{\xx}_{k}}

	\newcommand {\yy}   {\mathbf{y}}

	\newcommand {\zz}  {\mathbf{z}}

	\newcommand{\hf}{\frac12}

	\newcommand{\defeq}{\triangleq}

\begin{document}
	
	\title{MINRES: From Negative Curvature Detection to Monotonicity Properties}
	
	\author{Yang Liu\thanks{School of Mathematics and Physics, University of Queensland, Australia. Email: \tt{yang.liu2@uq.edu.au}}
		\and
		Fred Roosta\thanks{School of Mathematics and Physics, University of Queensland, Australia, and International Computer Science Institute, Berkeley, USA. Email: \tt{fred.roosta@uq.edu.au}}
	}
	
	\date{\today}
	\maketitle
	
	\abstract{The conjugate gradient method (CG) has long been the workhorse for inner-iterations of second-order algorithms for large-scale nonconvex optimization. Prominent examples include line-search based algorithms, e.g., Newton-CG, and those based on a trust-region framework, e.g., CG-Steihaug. This is mainly thanks to CG's several favorable properties, including certain monotonicity properties and its inherent ability to detect negative curvature directions, which can arise in nonconvex optimization. 
		This is despite the fact that the iterative method-of-choice when it comes to real symmetric but potentially indefinite matrices is arguably the celebrated minimal residual (MINRES) method. However, limited understanding of similar properties implied by MINRES in such settings has restricted its applicability within nonconvex optimization algorithms. We establish several such nontrivial properties of MINRES, including certain useful monotonicity as well as an inherent ability to detect negative curvature directions. These properties allow MINRES to be considered as a potentially superior alternative to CG for all Newton-type nonconvex optimization algorithms that employ CG as their subproblem solver. 
	}

	\section{Introduction}
\label{sec:intro}
Consider the linear least-squares problem 
\begin{align}
	\label{eq:ls}
	\min_{\xx \in \real^{d}} \vnorm{\AA \xx - \bb}^{2},
\end{align}
where $ \AA \in \real^{d \times d} $ is a symmetric but potentially indefinite and/or singular matrix and $ \bb \in \real^{d} $. Clearly, \cref{eq:ls} includes, as a special case, the symmetric linear system $\AA \xx = \bb$. Among many iterative algorithms designed for such a setting, e.g., \cite{golub2012matrix,paige1975solution,freund1994new,choi2011minres}, the minimum residual method (MINRES), introduced in the seminal work of Paige and Saunders \cite{paige1975solution}, is arguably the preferred Krylov subspace method. Even when $ \AA $ is positive definite, it has been shown that MINRES enjoys several advantageous properties that allow it to be considered a superior alternative to the primary workhorse, the conjugate gradient method (CG). For example, on positive definite system, it has been shown in \cite{fong2012cg} that, just like CG, not only does MINRES have monotonicity properties in terms of the iterate norms, but also both the Euclidean and the energy norms of the error are monotonic. However, in sharp contrast to CG, MINRES also provides monotonicity in the residual (and more general backward errors). In this light, when the stopping rule is based on the residual (or backward errors), \cite{fong2012cg} demonstrate that MINRES can often converge orders of magnitude faster than CG.

In the context of optimization of a twice continuously differentiable function $ f:\real^{d}\to \real $, subproblems of the form \cref{eq:ls} arise often where $ \AA $ and $ \bb $ are related to the Hessian and gradient of $ f $, and $ \xx $ is a direction for updating the optimization iterate, e.g., Newton's direction is given with $ \AA = \nabla^{2} f(\ww) $ and $ \bb = -\nabla f(\ww) $.
In such settings, the superior performance of MINRES over CG has further been verified in \cite{roosta2018newton,liu2021convergence}, where a variant of Newton's method employing MINRES as inner-solver, called Newton-MR, has been shown to greatly outperform the classical Newton-CG \cite{nocedal2006numerical} for optimization of strongly convex problems.

Despite such theoretical and empirical observations, the application of MINRES as subproblem solver for nonconvex Newton-type optimization methods has been very limited; an exception is \cite{roosta2018newton}, where nonconvexity is limited to invex optimization problems. This is in sharp contrast to CG, which has historically been used to calculate the update direction within many line-search and trust-region variants of second-order nonconvex optimization algorithms \cite{nocedal2006numerical,conn2000trust}. 
Beyond certain implied monotonicity properties, this is greatly thanks to CG's ability to naturally detect available nonpositive curvature (NPC) directions, i.e., some $ \vv \in \real^{d}, \vv \neq \zero  $, for which $ \dotprod{\vv,\AA\vv} \leq 0 $. Such directions provide a valuable tool within the optimization iterations to avoid entrapment in undesirable saddle points. 
In fact, beyond CG, all conjugate direction methods such as conjugate residual (CR) \cite{stiefel1955relaxationsmethoden,saad2003iterative} involve iterations that allow for ready access to NPC directions, when they arise. Such inherent ability has prompted researchers to design nonconvex Newton-type optimization algorithms, which leverage conjugate-direction methods as subproblem solvers and, by nontrivial use of NPC directions, come equipped with favorable convergence properties, e.g., \cite{dahito2019conjugate,royer2020newton,curtis2021trust,o2021log,xie2021complexity01,xie2021complexity02}. 

However, due to complicated dynamics of its iterations, similar properties for MINRES have not been available. In this paper, we set out to shed more light on MINRES and its underlying properties when applied to symmetric but potentially indefinite and/or singular systems. In doing so, not only do we establish several desirable monotonicity properties of MINRES, which mimic and at times improve those of CG, but also we show that MINRES indeed comes equipped with natural ability to detect NPC directions. It is hoped that such properties pave the way for further adoption of MINRES, as a potentially superior alternative to CG, within the class of nonconvex Newton-type optimization algorithms.

The rest of the paper is organized as follows. We end this section by introducing our notation and definitions as well as highlighting our main contributions. In \cref{sec:recall}, we briefly review MINRES and all its ingredients. We then provide our theoretical analysis of the properties of MINRES in \cref{sec:main}. This includes its inherent ability to detect NPC directions (\cref{sec:nc}) and several other useful properties such as monotonicity (\cref{sec:monotonicity}). Conclusion and further thoughts are gathered in \cref{sec:conclusion}.

\subsubsection*{Notation and definitions}
\label{sec:notation}
Throughout the paper, vectors and matrices are denoted by bold lower-case and bold upper-case letters, respectively, e.g., $ \bb $ and $ \AA $. 
We use regular lower-case letters to denote scalar constants, e.g., $ d $. 
For two real vectors $ \vv,\ww $, their inner-product is denoted by $ \dotprod{\vv, \ww}  = \vv^{\T} \ww $. For a vector $\vv$, its Euclidean norm is denoted by $ \|\vv\| $. 
The iteration counter for the main algorithm appears as a subscript. 
The zero vector is denoted by $\bm{0}$, while $ \ee_{j} $ denotes the $ j\th $ column of the identity matrix. 
The zero matrix is denoted by $ \zero $, while the identity matrix of dimension $ k \times k $ is denoted by $ \eye_{k} $. 
$\AA \succeq \zero$ indicate that the matrix $ \AA $ is positive semi-definite, while $ \AA \succ \zero $ denotes positive definiteness. 
Also, for a symmetric matrix $ \AA $, we denote $ \vnorm{\xx}_{\AA} = \xx^{\T} \AA \xx $ (this is an abuse of notation since unless $ \AA \succ \zero $, this does not imply a norm).
For any $ k\geq 1 $, $\mathcal{K}_{k}(\AA, \bb) = \Span\left\{\bb,\AA\bb,\ldots,\AA^{k-1}\bb\right\}$ denotes the Krylov subspace of degree $ k $ generated using $ \bb $ and $ \AA $.
The residual vector at iteration $ k $ is denoted by $ \rr_{k} = \bb - \AA \xx_{k} $. 

\begin{remark}[Initialization]
	\label{rem:initialize}
	For simplicity, we assume that $ \xx_{0} = \zero $ and hence $ \mathcal{K}_{k}(\AA, \rr_{0}) = \mathcal{K}_{k}(\AA, \bb) $. The arguments and derivations can be accordingly modified to account for a nonzero initialization.
\end{remark} 

Motivated by optimization applications where $ \AA $ is typically the Hessian encoding the geometry of the optimization landscape, \cref{def:nc} describes what we refer to as a nonpositive curvature direction. This simply amounts to any direction that lies in the eigenspace corresponding to the nonpositive eigenvalues of $ \AA $.
\begin{definition}[Nonpositive Curvature Direction]
	\label{def:nc}
	Any $ \vv \in \real^{d}, \; \vv \neq \zero  $ for which $ \dotprod{\vv,\AA\vv} \leq 0 $ is called a nonpositive curvature direction.
\end{definition} 

Clearly, the interplay between $ \AA $ and $ \bb $ is a critical factor in the convergence of MINRES and many other iterative solvers. This interplay is entirely characterized by the notion of the grade of $ \bb $ with respect to $ \AA $ as given in \cref{def:grade_b}; see \cite{saad2003iterative} for more details.
\begin{definition}
	\label{def:grade_b}
	The grade of $ \bb $ with respect to $ \AA $ is the positive integer $ g \in \mathbb{N} $ such that 
	\begin{align*}
		\text{dim}\left(\mathcal{K}_{k}(\AA, \bb)\right) = \begin{cases}
			k, &\quad k \leq g, \\
			g, &\quad k > g.
		\end{cases}
	\end{align*}
\end{definition}

\subsubsection*{Contributions}
Our main results can be summarized informally as follows.

\begin{informal*}
	\begin{enumerate}[label = \textbf{Property (\Roman*)}, wide]
		\item \label{item:nc} 
		As part of the MINRES iterations for solving \cref{eq:ls}, one can readily track a certain quantity without any additional cost (NPC Condition \cref{eq:c_gama_nc}) to determine if a nonpositive curvature is available at iteration $ k $. In this case, $ \rr_{k-1} $ (the residual vector of the \emph{previous} iteration) can be declared as such a direction (\cref{thm:determinant_T}). 
		\item \label{item:cert}
		Under a certain mild condition on the grade of $ \bb $ with respect to $ \AA $, when no nonpositive curvature direction is detected for all iterations, MINRES provides a certificate for positive semi-definiteness of $ \AA $ (\cref{thm:A_PSD}).
		
		\vspace{2mm}
		Moreover, as long as nonpositive curvature has not been detected, the MINRES iterates $ \{\xx_{k}\}_{k=1} $ have the following properties (\cref{thm:A_norm,thm:xTr_monotonicity}), in addition to several others to be detailed later:
		\item \label{item:descent} $\dotprod{\bb,\xx_{k}} > \dotprod{\xx_{k}, \AA \xx_{k}}$.
		\item \label{item:quadratic} $m(\xx_{k}) < m(\xx_{k-1})$, where $m(\xx) \triangleq \dotprod{\xx, \AA \xx}/2 - \dotprod{\bb,\xx}$.
		\item \label{item:norm} $\vnorm{\xx_{k}} > \vnorm{\xx_{k-1}}$.
	\end{enumerate}
\end{informal*}

\begin{remark}
	Consider an unconstrained optimization problem $ \min_{\ww \in \real^{d}} f(\xx) $ with a twice continuously differentiable objective $ f:\real^{d}\to \real $, and let $ \AA = \nabla^{2} f(\ww) $ and $ \bb = - \nabla f(\ww)$.
	Several observations can be immediately made to support our position in advocating the use of MINRES as a good subproblem solver within many Newton-type optimization algorithms. 
	\begin{itemize}
		\item At iteration $ k $ of MINRES, the residual vector $ \rr_{k-1} $ and quadratic value $ \dotprod{\rr_{k-1},\AA \rr_{k-1}} $ are both readily available without additional computational cost (\cref{lemma:rAr}). \labelcref{item:nc} guarantees that if a NPC direction is available at iteration $ k $, we must have $ \dotprod{\rr_{k-1},\AA \rr_{k-1}} \leq 0 $. That is, $ \rr_{k-1} $ is one such direction that can be readily used within the optimization algorithm. 
		\item Let $ f $ be a strongly convex function for which we have $\AA \succ \zero$. The iterates of CG in this case can be defined as 
		\begin{align*}
			\xx_{k} = \argmin_{\xx \in \mathcal{K}_{k}(\AA, \bb)} \; \hf \dotprod{\xx, \AA \xx} - \dotprod{\bb,\xx}.
		\end{align*}
		Since $\zero \in \mathcal{K}_{k}(\AA, \bb)$, the optimal value of the above quadratic must be nonpositive. This implies that $\dotprod{\bb,\xx_{k}} \geq \dotprod{\xx_{k}, \AA \xx_{k}}/2 > 0$.
		In other words, in Newton-CG applied to a strongly convex problem, all CG iterations amount to descent directions, i.e., $ \dotprod{\xx_{k},\nabla f(\ww)} < 0, \; \forall k $. From \labelcref{item:descent}, we see that not only does every iteration of MINRES amount to a descent direction, but also MINRES can produce a better direction in that the amount of descent can be greater than the equivalent direction obtained from CG. 
		\item \labelcref{item:quadratic,item:norm} of CG are the main driving force for its widespread use within trust-region algorithms, e.g., CG-Steihaug \cite{steihaug1983conjugate,conn2000trust}. Since MINRES also satisfies these properties, it can now be considered a viable alternative to CG for approximately solving the subproblems of the trust-region framework. 
	\end{itemize}
\end{remark}

\section{MINRES: Review}
\label{sec:recall}
For the sake of completeness, we recall the MINRES algorithm in some detail and introduce further notation that is used as part of our results. More details can be found in the original work \cite{paige1975solution} and many textbooks on the topic, e.g., \cite{bjorck2015numerical,demmel1997applied,golub2012matrix}. 

Recall that MINRES iterations, at a high level, can be written as 
\begin{align*}
	\xx_{k} = \argmin_{\xx \in \mathcal{K}_{k}(\AA, \bb)} \vnorm{\bb - \AA \xx}.
\end{align*}
At its core, MINRES involves three major ingredients, as depicted in \cref{alg:MINRES}: Lanczos process, QR decomposition, and update of the iterate.

\subsubsection*{Lanczos process}
With $ \vv_{1} = \bb/\vnorm{\bb} $, recall that after $k$ iterations of the Lanczos process, and in the absence of round-off errors\footnote{For our theoretical derivations here, we assume the absence of round-off errors. Of course, in practice, such errors will result in a loss of orthogonality in the vectors generated by the Lanczos process, which may be safely ignored in many applications. However, if strict orthogonality is required, it can be enforced by a reorthogonalization strategy.}, the Lanczos vectors form an orthogonal matrix $ \VV_{k+1} = \begin{bmatrix} \vv_1  & \vv_{2} & \dots &  \vv_{k+1} \end{bmatrix} \in \real^{d \times (k+1)} $, whose columns span $\mathcal{K}_{k+1}(\AA, \bb)$, and satisfy 
\begin{align}
	\label{eq:AV}
	\AA \VV_{k} = \VV_{k+1} \wTT_{k},
\end{align}
where $ \wTT_{k} \in \real^{(k+1) \times k}$ is an upper Hessenberg matrix of the form 
\begin{align}
	\label{eq:tridiagonal_T}
	\wTT_{k} &= 
	\begin{bmatrix}
		\alpha_1 & \beta_2 & & & \\
		\beta_2 & \alpha_2 & \beta_3 & & \\
		& \beta_3 & \alpha_3 & \ddots & \\
		& & \ddots & \ddots& \beta_{k} \\
		& & & \beta_{k} & \alpha_{k} \\
		\hdashline
		& & & & \beta_{k+1} \\
	\end{bmatrix}
	\triangleq
	\begin{bmatrix}
		\TT_{k} \\
		\beta_{k+1}\ee^{\T}_{k}
	\end{bmatrix}. 
\end{align}
Subsequently, we get the three-term recursion
\begin{align*}
	\AA \vv_{k} = \beta_{k} \vv_{k-1} + \alpha_{k} \vv_{k} + \beta_{k+1} \vv_{k+1}, \qquad k \geq 2.
\end{align*}
In this light, every iteration of MINRES requires exactly one matrix-vector product.
If $\xx_{k} = \VV_{k}\yy_{k}$ for some $ \yy_{k} \in \real^{k} $, the residual $ \rr_{k} $ can be written as 
\begin{align*}
	\rr_{k} = \bb - \AA \xx_{k} = \bb - \AA \VV_{k} \yy_{k} = \bb - \VV_{k+1} \wTT_{k} \yy_{k} = \VV_{k+1}(\vnorm{\bb} \ee_1 - \wTT_{k} \yy_{k}).
\end{align*}
This gives the well-known subproblems of MINRES as
\begin{align}
	\label{eq:MINRES_sub}
	\min_{\yy_{k} \in \real^{k}} \vnorm{\beta_{1} \ee_1 - \wTT_{k} \yy_{k}}, \qquad \beta_{1} = \vnorm{\bb}.
\end{align}

Recall that the grade of $ \bb $ with respect to $ \AA $, i.e., $ g $ in \cref{def:grade_b}, determines the maximum rank that the triangular matrix $ \TT_k $ can achieve, namely $ \max \left\{\rank(\TT_{k}); \; k \geq 0\right\} \leq g, \; 1 \leq k \leq d $. It is also well-known that, in exact arithmetic, the Lanczos process encounters the ``lucky breakdown'' after exactly $ g $ iterations, i.e., $ \beta_{g+1} = 0 $, in which case MINRES returns a solution to \cref{eq:ls}.

\subsubsection*{QR decomposition}
\label{sec:QR}
Recall that \cref{eq:MINRES_sub} is solved using the QR factorization of $ \wTT_{k} $. 
Let $ \QQ_{k} \wTT_{k} = \wRR_{k} $ be the full QR decomposition\footnote{For the sake of notational simplicity and to avoid overusing the ``transpose'' operator in the discussion that follows, we have given the full QR factorization of $ \wTT_{k} $ in a more unconventional way where the Q factor is on the left.} of $ \wTT_{k} $, where $ \QQ_{k} \in \real^{(k+1) \times (k+1)} $ and $ \wRR_{k} \in \real^{(k+1) \times k} $. Typically, $\QQ_{k}$ is formed, implicitly, by the application of a series of $ 2 \times 2 $ reflections to transform $ \wTT_{k} $ to the upper-triangular matrix $ \wRR_{k} $. Each reflection affects only two rows of the matrix being triangularized. More specifically, two successive reflections can be compactly written by considering the elements of the matrix that are being affected: 
\begin{align*}
	\begin{bmatrix}
		1 & 0 & 0 \\
		0 & c_{i-1} & s_{i-1} \\
		0 & s_{i-1} & -c_{i-1}
	\end{bmatrix} \begin{bmatrix}
		c_{i-2} & s_{i-2} & 0\\
		s_{i-2} & -c_{i-2} & 0 \\
		0 & 0 & 1 
	\end{bmatrix} 
	\begin{bmatrix}
		\gamma_{i-2}^{(1)} & \delta_{i-1}^{(1)} & 0 & 0 \\
		\beta_{i-1} & \alpha_{i-1} & \beta_{i} & 0 \\
		0 & \beta_{i} & \alpha_{i} & \beta_{i+1}
	\end{bmatrix}&
	\nonumber \\ 
	= \begin{bmatrix}
		1 & 0 & 0 \\
		0 & c_{i-1} & s_{i-1} \\
		0 & s_{i-1} & -c_{i-1}
	\end{bmatrix} \begin{bmatrix}
		\gamma_{i-2}^{(2)} & \delta_{i-1}^{(2)} & \epsilon_{i} & 0 \\
		0 & \gamma_{i-1}^{(1)} & \delta_{i}^{(1)} & 0 \\
		0 & \beta_{i} & \alpha_{i} & \beta_{i+1} \\
	\end{bmatrix}& \nonumber \\ 
	= \begin{bmatrix}
		\gamma_{i-2}^{(2)} & \delta_{i-1}^{(2)} & \epsilon_{i} & 0 \\
		0 & \gamma_{i-1}^{(2)} & \delta_{i}^{(2)} & \epsilon_{i+1} \\
		0 & 0 & \gamma_{i}^{(1)} & \delta_{i+1}^{(1)} \\
	\end{bmatrix}&,
\end{align*}
where $ 3 \leq i \leq k-1 $ and 
\begin{align}
	\label{eq:c_s_r2}
	c_{j}   = \frac{\gamma_{j}^{(1)}}{\gamma_{j}^{(2)}}, \quad s_{j} = \frac{\beta_{j+1}}{\gamma_{j}^{(2)}}, \quad \gamma_{j}^{(2)} = \sqrt{(\gamma_{j}^{(1)})^2 + \beta_{j+1}^2} = c_{j} \gamma_{j}^{(1)} + s_{j} \beta_{j+1}, \quad \quad 1 \leq j \leq k.
\end{align}
Here, the $ 2 \times 2 $ submatrix made of $c_{j}$ and $s_{j}$ is the special case of a Householder reflector in dimension two \cite[p.\ 76]{trefethen1997numerical}.

Consequently, we can rewrite $\QQ_{k}$ and $ \wRR_{k} $ in block form as 
\begin{subequations}
	\label{eq:T=QR}
	\begin{align}
		\QQ_{k} \wTT_{k} = \wRR_{k} \triangleq 
		\begin{bmatrix}
			\RR_{k}\\
			\zero^{\T}
		\end{bmatrix}, & \quad 
		\RR_{k} \triangleq 
		\begin{bmatrix}
			\gamma_{1}^{(2)} & \delta_2^{(2)} & \epsilon_3 & & & \\
			& \gamma_2^{(2)} & \delta_3^{(2)} & \epsilon_4 & &\\
			& & \gamma_3^{(2)} & \ddots & \ddots & \\
			& & & \ddots & \ddots & \epsilon_{k} \\
			& & & & \gamma_{k-1}^{(2)}  & \delta_{k}^{(2)} \\
			& & & & & \gamma_{k}^{(2)}
		\end{bmatrix}, \label{eq:block_R_tilde}\\
		\QQ_{k} \triangleq \QQ_{k,k+1} \begin{bmatrix}
			\QQ_{k-1} & \\
			& 1
		\end{bmatrix}, & \quad \QQ_{k,k+1} \triangleq 
		\begin{bmatrix}
			\eye_{k-1} & & \\
			& c_{k} & s_{k}\\
			& s_{k} & -c_{k}
		\end{bmatrix}&. \label{eq:block_Q} 
	\end{align}
\end{subequations}
In fact, the same series of transformations are simultaneously applied to $\beta_{1} \ee_1$ as 
\begin{align*}
	\QQ_{k}\beta_1\ee_1 = \beta_1
	\begin{bmatrix}
		c_1\\
		s_1 c_2\\
		\vdots\\
		s_1s_2\dots s_{k-1}c_{k}\\
		s_1s_2\dots s_{k-1}s_{k}
	\end{bmatrix} \triangleq
	\begin{bmatrix}
		\tau_1\\
		\tau_2\\
		\vdots\\
		\tau_{k}\\
		\phi_{k}
	\end{bmatrix} \triangleq \begin{bmatrix}
		\ttt_{k} \\
		\phi_{k}
	\end{bmatrix}.
\end{align*}

With these quantities available, we can solve \cref{eq:MINRES_sub} by noting that
\begin{align*}
	\min_{\yy_{k} \in \real^{k}}  \vnorm{\rr_{k}} &= \min_{\yy_{k} \in \real^{k}} \vnorm{\beta_1\ee_1 - \wTT_{k}\yy_{k}}   = \min_{\yy_{k} \in \real^{k}} \vnorm{\QQ_{k}\beta_1\ee_1 - \QQ_{k} \wTT_{k} \yy_{k}} = \min_{\yy_{k} \in \real^{k}} \vnorm{\begin{bmatrix}
			\ttt_{k}\\
			\phi_{k}
		\end{bmatrix} - 
		\begin{bmatrix}
			\RR_{k} \\
			\zero^{\T}
		\end{bmatrix}\yy_{k}},
\end{align*}
so that $\RR_{k} \yy_{k} = \ttt_{k}$ and $ \| \rr_{k} \| = \phi_{k} $ (note that, by construction, $ \phi_{k} $ remains non-negative throughout the iterations.) We also trivially have $ \phi_{0} = \beta_{1} = \vnorm{\bb} $.

\subsubsection*{Updates}
Suppose $ k < g $ and define $ \DD_{k} $ from the lower triangular system $ \RR^{\T}_{k} \DD^{\T}_{k} = \VV^{\T}_{k} $. Now, letting $ \VV_{k} = [ \VV_{k-1} \mid \vv_{k}] $ and using the fact that $ \RR_{k} $ is upper-triangular, we get the recursion $ \DD_{k} = [ \DD_{k-1} \mid \dd_{k}] $ for some vector $ \dd_{k} $. As a result, using $ \RR_{k} \yy_{k} = \ttt_{k} $, one can update the iterate as 
\begin{align*}
	\xx_{k} = \VV_{k} \yy_{k} = \DD_{k} \RR_{k} \yy_{k} = \DD_{k} \ttt_{k} = \begin{bmatrix}
		\DD_{k-1} & \dd_{k}
	\end{bmatrix} \begin{bmatrix}
		\ttt_{k-1} \\
		\tau_{k}
	\end{bmatrix} = \xx_{k-1} + \tau_{k} \dd_{k}.
\end{align*}
Furthermore, from $ \VV_{k} = \DD_{k} \RR_{k} $, i.e., 
\begin{align*}
	\begin{bmatrix}
		\vv_1 & \vv_2 & \ldots & \vv_{k}
	\end{bmatrix} = \begin{bmatrix}
		\dd_1 & \dd_2 & \ldots & \dd_{k}
	\end{bmatrix} \begin{bmatrix}
		\gamma_1^{(2)} & \delta_2^{(2)} & \epsilon_3 & & \\
		& \gamma_2^{(2)} & \ddots & \ddots & \\
		& & \ddots & \ddots & \epsilon_{k} \\
		& & & \gamma_{k-1}^{(2)}  & \delta_{k}^{(2)} \\
		& & & & \gamma_{k}^{(2)}
	\end{bmatrix},
\end{align*}
we get the following relationship for computing $ \vv_{k} $ as
\begin{align*}
	\vv_{k} = \epsilon_{k} \dd_{k-2} + \delta^{(2)}_{k} \dd_{k-1} + \gamma^{(2)}_{k} \dd_{k}.
\end{align*}
All of the above steps constitute MINRES algorithm, which is given in \cref{alg:MINRES}. For more details, see \cite{bjorck2015numerical,demmel1997applied,golub2012matrix,choi2011minres,choi2014algorithm,fong2012cg,paige1975solution}. 

\begin{algorithm}
	\caption{\textbf{MINRES with Built-in NPC Detection}}
	\begin{algorithmic}
		\label{alg:MINRES}
		\STATE \textbf{Input:} The matrix $ \AA $, the right hand side vector $\bb$,
		\vspace{1mm}
		\STATE $ \beta_1 = \vnorm{\bb} $, $ \rr_{0} = \bb $, $ \vv_1 = \bb /\beta_1 $, $ \vv_0 = \xx_0 = \dd_0 = \dd_{-1} = \zero $, $ c_0 = -1 $, $ s_0 = 0 $, $ \phi_0 = \tau_0 = \beta_1 $, $ \delta^{(1)}_1 = 0 $, $ k = 1 $,
		\vspace{1mm}
		\WHILE {True}
		\vspace{1mm}
		\STATE $ \pp_{k} = \AA \vv_{k}$, $\alpha_{k} = \vv^{\T}_{k} \pp_{k} $, $ \pp_{k} = \pp_{k} - \beta_{k} \vv_{k-1} $, $ \pp_{k} = \pp_{k} - \alpha_{k} \vv_{k} $, $ \beta_{k+1} = \vnorm{\pp_{k}} $,
		\vspace{1mm}
		\STATE $ \delta^{(2)}_{k} = c_{k-1} \delta^{(1)}_{k} + s_{k-1} \alpha_{k} $, $ \gamma^{(1)}_{k} = s_{k-1} \delta^{(1)}_{k} - c_{k-1} \alpha_{k} $, $ \epsilon_{k+1} = s_{k-1} \beta_{k+1} $,  $ \delta^{(1)}_{k+1} = -c_{k-1} \beta_{k+1} $,
		\vspace{1mm}
		\IF{$ c_{k-1} \gamma^{(1)}_{k} \geq 0$} 
		\vspace{1mm}
		\RETURN $ \rr_{k-1} $ \text{as a NPC direction} (\textbf{NB:} this is an optional return; see \cref{rem:nc_return}),
		\vspace{1mm}
		\ENDIF
		\vspace{1mm}
		\STATE $ \gamma_{k}^{(2)} = \sqrt{(\gamma_{k}^{(1)})^2 + \beta_{k+1}^2} $,
		\vspace{1mm}
		\IF{ $ \gamma^{(2)}_{k} \neq 0 $}
		\vspace{1mm}
		\STATE $c_{k} = \gamma_{k}^{(1)} / \gamma_{k}^{(2)} $, $ s_{k} = \beta_{k+1} / \gamma_{k}^{(2)} $, $ \tau_{k} = c_{k} \phi_{k-1} $, $ \phi_{k} = s_{k} \phi_{k-1} $,
		\vspace{1mm}
		\STATE $ \dd_{k} = \left(\vv_{k} - \delta^{(2)}_{k} \dd_{k-1} - \epsilon_{k} \dd_{k-2} \right) / \gamma^{(2)}_{k} $, $ \xx_{k} = \xx_{k-1} + \tau_{k} \dd_{k} $,
		\vspace{1mm}
		\IF{$ \beta_{k+1} \neq 0 $} 
		\vspace{1mm}
		\STATE $ \vv_{k+1} = \pp_{k} / \beta_{k+1} $, $ \rr_{k} = s_{k}^2 \rr_{k-1} - \phi_{k} c_{k} \vv_{k+1} $,
		\vspace{1mm}
		\ELSE
		\vspace{1mm}
		\STATE $ \rr_{k} = \zero $,
		\RETURN $ \xx_{k} $ as a solution to \cref{eq:ls}.
		\vspace{1mm}
		\ENDIF
		\ELSE
		\vspace{1mm} 
		\STATE $ c_{k} = 0 $, $ s_{k} = 1 $, $ \tau_{k} = 0 $, $ \phi_{k} = \phi_{k-1} $, $ \rr_{k} = \rr_{k-1} $, $ \xx_{k} = \xx_{k-1} $,
		\vspace{1mm}
		\RETURN $ \xx_{k} $ as a solution to \cref{eq:ls}.
		\vspace{1mm}
		\ENDIF
		\vspace{1mm}
		\STATE $ k \leftarrow k+1 $,
		\vspace{1mm}
		\ENDWHILE
		\vspace{1mm}
		\STATE \textbf{Output:} A solution to \cref{eq:ls} or a NPC direction.
	\end{algorithmic}
\end{algorithm}

\begin{remark}
	\label{rem:nc_return}
	\cref{alg:MINRES} is depicted such that the iterations are terminated when an NPC direction is detected. However, depending on the application, one may choose to continue the MINRES iterations to additionally find a solution, or a suitable approximation, to \cref{eq:ls}. In other words, while the detection of NPC direction is a feature that can be leveraged in many applications, it does not necessarily imply a termination to the MINRES iterations.
\end{remark}

\section{MINRES: Main properties}
\label{sec:main}
In this section, we provide several useful properties of MINRES when applied to symmetric but potentially indefinite and/or singular systems. These include its natural ability to detect NPC directions (\cref{sec:nc}) as well as many monotonicity properties (\cref{sec:monotonicity}).

\subsection{Nonpositive curvature detection}
\label{sec:nc}
First, we show that MINRES comes naturally equipped with the ability to detect nonpositive curvature. More specifically, we show that whenever a certain readily verifiable condition within MINRES iteration holds, the residual of the previous iteration is one such direction.

\subsubsection*{Certificate of positive-definiteness of $ \TT_{k} $}
Since $ \rr_{k-1} \in \mathcal{K}_{k} (\AA, \bb) $, we can write $ \rr_{k-1} = \VV_{k} \zz $ for some $ \zz \in \real^{k} $ and $ \VV_{k} \in \real^{d \times k} $ in \cref{eq:AV}. For $ \rr_{k-1} $ to be a direction of nonpositive curvature, we must have $ \rr_{k-1}^{\T} \AA \rr_{k-1} = \zz^{\T} \TT_{k} \zz \leq 0 $, i.e., we must have $ \TT_{k} \not \succ \zero $. We show that the converse also holds (\cref{thm:determinant_T}). In other words, as soon as $ \TT_{k} \not \succ \zero $ within MINRES iterations, we must have $ \rr_{k-1}^{\T} \AA \rr_{k-1} \leq 0 $. This amounts to showing that the sequence $ \left\{\rr_{i-1}^{\T} \AA \rr_{i-1}; \; i = 1, \ldots, k \right\}$ provides a built-in certificate of positive-definiteness for $ \TT_{k} $. 

We first show that $ \rr_{i-1}^{\T} \AA \rr_{i-1} $ can be computed without additional matrix-vector products. Indeed, among other useful relations, \cref{lemma:rAr} gives an equivalent expression for $ \rr_{i-1}^{\T} \AA \rr_{i-1} $ that can be readily computed using scalar operations. 

\begin{lemma}
	\label{lemma:rAr}
	Let $ g $ be the grade of $ \bb $ with respect to $ \AA $ as in \cref{def:grade_b}. We have
	\begin{subequations}        
		\begin{align}
			\xx_{i}^{\T} \AA \rr_{k} &= 0, \quad 1 \leq i \leq k \leq g, \label{eq:xAr_perp} \\
			\rr_{i}^{\T} \AA \rr_{k} &= 0, \quad 1 \leq i \leq g, ~1 \leq k \leq g, ~i \neq k, \label{eq:rAr_perp} \\
			\rr_{k-1}^{\T} \AA \rr_{k-1} &= -\phi_{k-1}^2 c_{k-1} \gamma^{(1)}_{k}, \quad 1 \leq k \leq g, \label{eq:rAr} \\        
			\rr_{k}^{\T} \bb &= \vnorm{\rr_{k}}^2, \quad 1 \leq k \leq g. \label{eq:bTrk}
		\end{align}
	\end{subequations}
\end{lemma}
\begin{proof}
	First, recall that $ \xx_{0} = \zero $ implies $ \xx_{i} \in \mathcal{K}_{k}(\AA, \bb), \; i \leq k $. From Petrov--Galerkin conditions, we also have $\rr_{k} \perp \AA \mathcal{K}_{k}(\AA, \bb)$, which with $ \AA $ real symmetric implies \cref{eq:xAr_perp}.
	Now let $ i < k $, for which we have $ \rr_{i} \in \mathcal{K}_{k}(\AA, \bb) $. From $ \AA \rr_{k} \perp \mathcal{K}_{k}(\AA, \bb) $, we get $ \rr_{i}^{\T} \AA \rr_{k} = 0 $ for $ i < k $. From the symmetry of $ \AA $, we get $ \rr_{i}^{\T} \AA \rr_{k} = 0 $ for all $ i \neq j $, which gives \cref{eq:rAr_perp}. Now, from \cite{choi2011minres,choi2014algorithm} we know that in MINRES, $ \rr_{k} $ and $ \AA \rr_{k} $ satisfy
	\begin{align}
		\label{eq:rk_rec}
		\rr_{k} = s_{k}^2 \rr_{k-1} - \phi_{k} c_{k} \vv_{k+1}, \quad \AA \rr_{k} = \phi_{k} (\gamma^{(1)}_{k+1} \vv_{k+1} + \delta^{(1)}_{k+2} \vv_{k+2}).
	\end{align}
	Noting that $ \rr_{k-2} \perp \Span\left\{\vv_{k},\vv_{k+1}\right\} $, and $ \vv_{k} \perp \vv_{k+1} $, we have
	\begin{align*}
		\rr_{k-1}^{\T} \AA \rr_{k-1} = \left( s_{k-1}^2 \rr_{k-2} - \phi_{k-1} c_{k-1} \vv_{k} \right)^{\T} \left( \phi_{k-1} (\gamma^{(1)}_{k} \vv_{k} + \delta^{(1)}_{k+1} \vv_{k+1}) \right) = - \phi_{k-1}^2 c_{k-1} \gamma^{(1)}_{k},
	\end{align*}
	which gives \cref{eq:rAr}. Finally, from \cref{eq:xAr_perp} we get 
	\begin{align*}
		\rr_{k}^{\T} \bb = \rr_{k}^{\T} \left( \rr_{k} + \AA \xx_{k} \right) = \vnorm{\rr_{k}}^2,
	\end{align*}
	which yields \cref{eq:bTrk}.
\end{proof}

\begin{remark}[NPC Condition]
	\label{rem:MINRES_NC}
	From \cref{eq:rAr}, we see that if 
	\begin{align}
		\label{eq:c_gama_nc}
		c_{k-1} \gamma^{(1)}_{k} \geq 0,
	\end{align}
	then $ \rr_{k-1} $ is a NPC direction for $ \AA $, i.e.,
	\begin{align*}
		\frac{\dotprod{\rr_{k-1},\AA \rr_{k-1}}}{\vnorm{\rr_{k-1}}^2} = -c_{k-1} \gamma^{(1)}_{k} \leq 0.
	\end{align*}
\end{remark}

Consider any $ \ell \leq g $. Since $\Span \{ \rr_{0}, \ldots, \rr_{k-1}\}  \subseteq \mathcal{K}_{k} (\AA, \bb) = \range(\VV_{k}) $, we have $\rr_{k-1} = \VV_{k} \yy$ for some $ \yy \in \real^{k} $. From $ \TT_{k} = \VV_{k}^{\T} \AA \VV_{k} $, it follows that $ \rr_{k-1}^{\T} \AA \rr_{k-1} = \yy^{\T} \TT_{k} \yy $. As a result, if $ \TT_{\ell} \succ \zero $,  then \cref{eq:rAr} implies that the NPC condition \cref{eq:c_gama_nc} does not hold for any $ k \leq \ell $. In other words, if $\TT_{\ell} \succ \zero$, then $c_{k-1} \gamma^{(1)}_{k} < 0, \; \forall 1\leq k \leq \ell$.
\cref{thm:determinant_T} below provides a converse to this. In particular, it states that as soon as $ \TT_{k} \not \succ \zero $, the NPC condition \cref{eq:c_gama_nc} holds and $ \rr_{k-1} $ is identified as a NPC direction. That is, \cref{thm:determinant_T} states that if $c_{k-1} \gamma^{(1)}_{k} < 0, \; \forall 1\leq k \leq \ell$, then we must have $\TT_{\ell} \succ \zero$.

\begin{theorem}[Certificate for $ \TT_{k} \succ \zero $]
	\label{thm:determinant_T}
	Suppose $ \ell \leq g $ where $ g $ is the grade of $ \bb $ with respect to $ \AA $ as in \cref{def:grade_b}. If $ \TT_{\ell} \not \succ \zero $, then the NPC condition \cref{eq:c_gama_nc} holds for some $ k \leq \ell $. In particular, if $ k \leq g $ is the first iteration where $ \TT_{k} \not \succ \zero $, then the NPC condition \cref{eq:c_gama_nc} holds.
\end{theorem}
\begin{proof}
	Suppose $ \TT_{\ell} \not \succ \zero $. We first note that $ k \leq g $ implies that $\rr_{k-1} \neq \zero$, which in turn gives $ \phi_{k-1} \neq 0 $. From the properties of the Krylov subspace, we have $ \Span \{ \rr_{0}, \ldots, \rr_{\ell-1}\} \subseteq \mathcal{K}_{\ell} (\AA, \bb) $. We now consider two cases.
	
	\begin{enumerate}[label = {\bfseries (\roman*)}]
		\item \label{case:2a} Let's first consider the case where $ \Span \{ \rr_{0}, \ldots, \rr_{\ell-1}\} = \mathcal{K}_{\ell} (\AA, \bb) $. 
		Suppose for any nonzero $\zz \in \mathcal{K}_{\ell} (\AA, \bb)$, we have $\dotprod{\zz, \AA \zz} > 0$. Consequently, for any nonzero $ \ww \in \real^{\ell} $, we can let $ \zz = \VV_{\ell} \ww \in \mathcal{K}_{\ell} (\AA, \bb) $ and have
		\begin{align*}
			\ww^{\T} \TT_{\ell} \ww = \ww^{\T} \VV_{\ell}^{\T} \AA \VV_{\ell} \ww = \zz^{\T} \AA \zz > 0. 
		\end{align*}
		However, this implies that $ \TT_{\ell} \succ \zero $, which contradicts the assumption on $ \TT_{\ell} $. Hence, there must exists nonzero $ \zz \in \mathcal{K}_{\ell} (\AA, \bb) $ for which $ \zz^{\T} \AA \zz \leq 0 $. Let $ \xi_{1}, \ldots, \xi_{\ell} $ be scalars such that 
		\begin{align*}
			\zz = \sum_{k=1}^{\ell} \xi_{k} \rr_{k-1}.
		\end{align*}
		Suppose $ \rr_{k-1}^{\T} \AA \rr_{k-1} > 0 $ for all $ 1 \leq k \leq \ell $. Then by \cref{eq:rAr_perp}, we have
		\begin{align*}
			0 \geq \zz^{\T} \AA \zz = \sum_{i=1}^{\ell} \sum_{j=1}^{\ell} \xi_{i} \xi_{j} \rr_{i-1}^{\T} \AA \rr_{j-1} = \sum_{k=1}^{\ell} \xi_{k}^2 \rr_{k-1}^{\T} \AA \rr_{k-1} > 0,
		\end{align*}
		which is a contradiction. Hence, we must have $ \rr_{k-1}^{\T} \AA \rr_{k-1} \leq 0 $ for some $ 1\leq k \leq \ell $, i.e., the NPC condition \cref{eq:c_gama_nc} holds for some $ k \leq \ell $. In particular, if $ k $ is the first iteration where $ \TT_{k} \not \succ \zero $, we must in fact have $ \rr_{k-1}^{\T} \AA \rr_{k-1} \leq 0 $. In other words, as soon as $ \TT_{k} $ contains a nonpositive eigenvalue, the NPC condition \cref{eq:c_gama_nc} holds and hence the vector $ \rr_{k-1} $ is a nonpositive curvature direction. 
		
		\item Now suppose $ \Span \{ \rr_{0}, \ldots, \rr_{\ell-1}\} \subsetneq \mathcal{K}_{\ell} (\AA, \bb) $. Suppose $ k \leq \ell $ is the first iteration where the additional residual vector has not increased the Krylov space's dimension, i.e., $\Span \{ \rr_{0}, \ldots, \rr_{k-1}\} = \mathcal{K}_{k-1} (\AA, \bb)$. 
		In this case, since $ \Span\{\rr_{k-2},\rr_{k-1}\} \perp \vv_{k} $, by \cref{eq:rk_rec} we must have $\phi_{k-1} c_{k-1} =0 $, i.e., $ \rr_{k-1} = s_{k-1}^2 \rr_{k-2} $. Since $ \phi_{k-1} \neq 0 $, we can only have $ c_{k-1} = 0 $ and $ s_{k-1} = 1 $, i.e., $ \rr_{k-1} = \rr_{k-2} $. 
		Hence, by \cref{eq:rAr}, we get
		\begin{align*}
			\rr_{k-1}^{\T} \AA \rr_{k-1} = \rr_{k-2}^{\T} \AA \rr_{k-2} = 0,
		\end{align*}
		which implies that both $ \rr_{k-1}$ and $\rr_{k-2} $ are nonpositive curvature directions (this also implies that $ \TT_{k-1} \not \succ \zero $). In fact, since $ \Span \{ \rr_{0}, \ldots, \rr_{k-2}\} = \mathcal{K}_{k-1} (\AA, \bb) $, based on the discussion of the previous case, having $ \rr_{k-2}^{\T} \AA \rr_{k-2} = 0 $ implies that we should have already detected nonpositive curvature in the previous iteration.
	\end{enumerate}
\end{proof}

It is possible to show that if $ \AA \succeq \zero$, then for any $ 1 \leq k \leq g - 1 $, we have $ \rr_{k-1}^{\T} \AA \rr_{k-1} > 0 $ and $ \TT_{k} \succ \zero$; see \cref{lemma:Lanczos} below. More generally, however, for any real symmetric but not necessarily PSD matrix $ \AA $, \cref{thm:determinant_T} implies that by keeping track of the NPC condition \cref{eq:c_gama_nc}, if $ \rr_{i-1}^{\T} \AA \rr_{i-1} > 0$ for all $1 \leq i \leq k \leq g $, then $ \TT_{k} \succ \zero $. In other words, MINRES has a built-in mechanism to provide a certificate of positive definiteness (or lack thereof) for $ \TT_{k} $. 

\subsubsection*{Certificate for $ \AA \succeq \zero $}
A natural question to ask is whether it is possible to provide a similar certificate for $ \AA $ from within MINRES iterations? 
The question arises because, as part of the Lanczos process, one expects to see tight connections between the spectral properties of $ \AA $ and $ \TT_{k} $. 
For examples, when $ \AA \succ \zero $, we trivially have $ \TT_{k} \succ \zero, \; 1 \leq k \leq g$. In fact, having $ \AA \succeq \zero$ implies positive definiteness of $ \TT_{k} $ for $ 1 \leq k \leq g-1 $. 
%

\begin{lemma}
	\label{lemma:Lanczos}
	Suppose $ \AA \succeq \zero$ and let $ g $ be the grade of $ \bb $ with respect to $ \AA $ as in \cref{def:grade_b}. The following statements hold.
	\begin{enumerate}[label = \textbf{(\roman*)}]
		\item $ \TT_{k} \succ \zero, \; 1 \leq k \leq g-1$.
		\item If $ \bb \in \range(\AA) $, then $ \TT_{g} \succ \zero$.
		\item If $ \bb \notin \range(\AA) $, then $ \TT_{g} \succeq \zero $ is singular, $ \gamma_{g}^{(2)} = 0 $, and $ \rr_{g-1} $ is a zero curvature direction.
	\end{enumerate}
\end{lemma}
\begin{proof}
	From $ \TT_{k} = \VV_{k}^{\T} \AA \VV_{k} $ and $ \AA \succeq \zero $, it follows that $ \TT_{k} \succeq \zero, 1 \leq k \leq g $. By the Sturm Sequence Property and the strict interlacing result \cite[Theorem 8.4.1]{golub2012matrix}, the smallest eigenvalue of $ \TTt $ decreases to zero monotonically with $ k $. As a result, only $ \TT_{g} $ can potentially be singular in which case it will have exactly one zero eigenvalue. In other words, we must always have $ \TT_{k} \succ \zero, \; 1 \leq k \leq g-1 $. 
	
	Now, if $ \bb \in \range(\AA) $, then by \cite[Property J2, p.~247]{hnvetynkova2007lanczos}, we know that $ \TT_{g} \succ \zero$. Otherwise, from \cite[Theorem 3.2]{choi2011minres}, it follows that $ \TT_{g} $ is singular and $ \gamma_{g}^{(1)} = 0 $. Hence, by \cref{eq:rAr}, we have $ \rr_{g-1}^{\T} \AA \rr_{g-1} =  0 $. In other words, $ \rr_{g-1} $ is a zero curvature direction. Furthermore, since $ \beta_{g+1} = 0 $,  we also get $ \gamma_{g}^{(2)} = 0 $ by \cref{eq:c_s_r2}. 
\end{proof}

From \cref{lemma:Lanczos}, it follows that when $ \AA \succeq \zero$, we have $ \TT_{k} \succ \zero, \;  1 \leq k \leq g - 1$, and $ \TT_{g} \succeq \zero$ (if $ \bb \in \range(\AA) $, we actually have $ \TT_{g} \succ \zero$). \cref{thm:A_PSD} below provides somewhat of a converse to this. More specifically, \cref{thm:A_PSD} shows that, under a certain condition on $ g $, having $ \TT_{k} \succ \zero,\; 1\leq k \leq g $, provides a certificate for $ \AA \succeq \zero$.
As a result, \cref{thm:A_PSD} shows that MINRES comes equipped with an inherent ability to provide a certificate of positive semi-definiteness for $ \AA $. This, just like the case for $ \TT_{k} $, can be done by tracking the NPC condition \cref{eq:c_gama_nc}. In doing so, however, one has to also take into account the interplay of $ \AA $ and $ \bb $ as it relates to $ g $, i.e., the grade of $ \bb $ with respect to $ \AA $ (\cref{def:grade_b}). Indeed, consider the case where $ \AA $ is indefinite, but $ \bb $ is an eigenvector corresponding to one of its positive eigenvalues. In this case, since $ g =1 $, MINRES terminates after only one iteration and the negative spectrum of $ \AA $ is never explored. To avoid such situations, we need $ \bb $ to have a non-zero projection on the eigenspace corresponding to each eigenvalue of $ \AA $, which is equivalent to $ g $ being equal to the number of distinct eigenvalues of $ \AA $.

\begin{theorem}[Certificate for $ \AA \succeq \zero $]
	\label{thm:A_PSD}
	Assume $ g $ equals the number of distinct eigenvalues of  $ \AA $, where $ g $ is the grade of $ \bb $ with respect to $ \AA $  as in \cref{def:grade_b}. If $ \AA \not \succeq \zero $, then $\TT_{k} \not \succ \zero$ for some $k \leq g$, and hence the NPC condition \cref{eq:c_gama_nc} holds for some $ k \leq g $. 
\end{theorem}

\begin{proof}
	Suppose $ \AA \not \succeq \zero $. If $ k \leq g $ is the first iteration such that $ \TT_{k} \not \succ \zero $, then by \cref{thm:determinant_T}, we have that the NPC condition \cref{eq:c_gama_nc} holds. Now suppose, for all $ k \leq g $, we always have $ \TT_{k} \succ \zero $. In this case, in particular, we have $ \TT_{g}  = \VV_{g}^{\T} \AA \VV_{g} \succ \zero $. Recall that the assumption on $ g $ is equivalent to $ \bb $ having a non-zero projection on the eigenspace corresponding to each eigenvalue of $ \AA $. Hence, using a very similar line of reasoning as in the proof of \cite[Theorem 2.6.2]{bjorck2015numerical}, we can show that $\range(\VV_{g})$ contains an eigenvector corresponding to a negative eigenvalue of $ \AA $. Let $ \vv $ be such an eigenvector, so we get $\dotprod{\vv, \AA \vv} < 0$. Since $ \vv \in \range(\VV_{g}) $, we can write $ \vv = \VV_{g} \ww $ for some $ \ww \in \real^{g} $. Now, it follows that
	\begin{align*}
		\ww^{\T} \TT_{g} \ww = \ww^{\T} \VV_{g}^{\T} \AA \VV_{g} \ww = \vv^{\T} \AA \vv < 0, 
	\end{align*}
	which contradicts having $ \TT_{g} \succ \zero $. Hence, we must have $ \TT_{g} \not \succ \zero $, which by \cref{thm:determinant_T} implies that the NPC condition \cref{eq:c_gama_nc} holds with $ k = g $.
\end{proof}

\begin{remark}
	In cases where one's primary objective is either to obtain a certificate of positive semi-definiteness or a direction of nonpositive curvature for $ \AA $, one can draw $ \bb $ from a suitably chosen random distribution, e.g., uniform distribution on the unit sphere. In this case, $ g $ satisfies the assumption of \cref{thm:A_PSD} with probability one.
\end{remark}

\subsubsection*{Top right $ (k-1)\times(k-1) $ block of $\RR_{k}$}
\cref{thm:determinant_T} provides a characterization for the positive-definiteness of $\TT_{k}$ in terms of the NPC condition \cref{eq:c_gama_nc}. 
Recall that, by Sylvester's criterion, $ \TT_{k} \succ \zero $ is a necessary and sufficient condition for the determinant of all trailing principal minors of $ \TT_{k} $ to be positive.
In this light, and as a corollary of \cref{thm:determinant_T}, the NPC condition \cref{eq:c_gama_nc} provides a characterization for the sign of such determinants.
More specifically, let the determinant of the trailing principal minors of $ \TT_{k} $ be defined as
\begin{align}
	p_{(k,l)} &\defeq \det \TT_{(k,l)} \triangleq \det \begin{bmatrix}
		\alpha_{k-l+1} & \beta_{k-l+2} & & \\
		\beta_{k-l+2} & \ddots & \ddots & \\
		& \ddots & \ddots& \beta_{k} \\
		& & \beta_{k} & \alpha_{k}
	\end{bmatrix}, \qquad 1\leq l \leq k \leq g, \label{eq:det_T_R:T}
\end{align}
where $ g $ is the grade of $ \bb $ with respect to $ \AA $.
\cref{thm:determinant_T} implies that as long as the NPC condition \cref{eq:c_gama_nc} has not been detected after $ k $ iterations of MINRES, we have $ p_{(k,l)} > 0, \;  l = 1, \ldots, k $.  

Recall the upper-triangular matrix $ \RR_{k} $ in \cref{eq:block_R_tilde}. Clearly, by \cref{eq:c_s_r2}, we trivially get that the determinant of the trailing principal minors of $ \RR_{k}, \; 1 \leq k \leq g $, is also positive. It turns out that we can also make a similar statement about the top right $ (k-1) \times (k-1) $ block of $ \RR_{k} $, namely 
\begin{align*}
	\mathbf{S}_{k} =
	\begin{bmatrix}
		\delta_2^{(2)} & \epsilon_3 & & \\
		\gamma_2^{(2)} & \delta_3^{(2)} & \epsilon_4 & \\
		& \gamma_3^{(2)} & \ddots & \ddots & \\
		& & \ddots & \ddots & \epsilon_{k} \\
		& & & \gamma_{k-1}^{(2)}  & \delta_{k}^{(2)} \\
	\end{bmatrix}&.
\end{align*}

\begin{theorem}
	\label{thm:determinants}     
	Let the determinant of the trailing principal minors of $ \mathbf{S}_{k} $ be defined as
	\begin{align}
		q_{(k,l)} &\defeq \det \begin{bmatrix}
			\delta_{k-l+1}^{(2)} & \epsilon_{k-l+2} & & \\
			\gamma_{k-l+1}^{(2)} & \ddots & \ddots & \\
			& \ddots & \ddots& \epsilon_{k} \\
			& & \gamma_{k-1}^{(2)} & \delta_{k}^{(2)}
		\end{bmatrix}, \qquad 1\leq l < k \leq g,  \label{eq:det_T_R:S}
	\end{align}
	where $ g $ is the grade of $ \bb $ with respect to $ \AA $ as in \cref{def:grade_b}.
	As long as the NPC condition \cref{eq:c_gama_nc} has not been detected for $ 2 \leq k \leq g $, we must have 
	\begin{align*}
		q_{(k,l)} > 0, \quad 1 \leq l < k \leq g.
	\end{align*}
\end{theorem}
\begin{proof}
	We give an outline of the proof. The detailed proof is technical and can be found in \cref{sec:appendix}. The proof makes use of the three-term recurrence relation for the determinant of a tridiagonal matrix (e.g.,  \cite[\S 0.9.10]{horn2012matrix} and \cite[Theorem 2.1]{el2004inverse}), namely 
	\begin{align}
		\label{eq:det_iters}
		q_{(k,l)} = \delta_{k-l+1}^{(2)}q_{(k,l-1)} - \gamma_{k-l+1}^{(2)} \epsilon_{k-l+2} q_{(k,l-2)}, \quad l = 1, \ldots, k - 1,
	\end{align}
	where we have set $ q_{(k,-1)} = 0$ and $q_{(k,0)} = 1 $. 
	We then consider the following ansatz for $ q_{(k,l)} $,
	\begin{align}
		\label{eq:det_R_formula}
		q_{(k,l)} = \left( \frac{p_{(k,l)}}{\prod_{i=1}^{l} \gamma_{k-i}^{(2)}} + c_{k-l-1} \sum_{i=1}^{l} \frac{(-1)^{l-i+1} \gamma_{k-i}^{(1)} p_{(k,i-1)}}{\prod_{j=1}^{i} \gamma_{k-j}^{(2)}} \right) \prod_{i=1}^{l} \beta_{k-i+1},
	\end{align}
	where $ p_{(k,l)} $ is defined as in \cref{eq:det_T_R:T}, and we set $ p_{(k,-1)} = 0$ and $p_{(k,0)} = 1 $.    The proof follows by showing that \cref{eq:det_R_formula} is indeed a correct solution to \cref{eq:det_iters}, which is done by induction. The last part of the proof involves showing that \cref{eq:det_R_formula} is positive by establishing the sign of the individual terms in the sum.
\end{proof}

Although \cref{thm:determinants} is a crucial ingredient in the proof of the main results of \cref{sec:monotonicity}, it can be in fact of independent interest.

\subsection{Monotonicity properties}
\label{sec:monotonicity}

When applied to positive definite systems, CG enjoys several monotonicity properties, which have historically motivated its use as subproblem solver within the optimization literature. For example, the fact that, starting with $ \xx_{0} = \zero $, CG iterates have increasing Euclidean norm while monotonically decreasing the quadratic $ \dotprod{\xx,\AA\xx}/2 - \dotprod{\xx,\bb} $ has been heavily leveraged within the trust-region framework, e.g., CG-Steihaug method within trust-region framework  \cite{steihaug1983conjugate,conn2000trust} with $ \AA $ and $ -\bb $ being the Hessian and the gradient of the objective function, respectively. For the special case where $ \AA \succ \zero $, \cite{fong2012cg} provides such monotonicity properties for MINRES. This is done indirectly by analyzing CR and leveraging the fact that, when $ \AA \succ \zero $, MINRES and CR are in fact equivalent algorithms. In this section, we extend those results to the case where $ \AA $ is any symmetric and possibly indefinite and/or singular matrix. By direct analysis of the MINRES algorithm, we show that, as long as nonpositive curvature has not been detected, MINRES satisfies very similar monotonicity properties, which can be useful within various optimization frameworks. The results of this section depend on several technical lemmas, which are given in \cref{sec:appendix}.

\cref{thm:xTr_monotonicity} sheds light on the sign and monotonicity of various relevant quantities of interest. For example, from the construction of CG, when $\AA \succ \zero$, we always have that the quadratic $ \dotprod{\xx,\bb} - \dotprod{\xx,\AA\xx}/2 $ remains positive. \cref{thm:xTr_monotonicity} shows that, as long as the NPC condition \cref{eq:c_gama_nc} has not been detected within MINRES, we can expect to see a similar but meaningfully different (\cref{rem:qaud}) quadratic to remain positive.

\begin{theorem}
	\label{thm:xTr_monotonicity}
	Let $ g $ be the grade of $ \bb $ with respect to $ \AA $  as in \cref{def:grade_b}. As long as the NPC condition \cref{eq:c_gama_nc} has not been detected for $1 \leq k < g $, we have
	\begin{subequations}
		\label{eq:xTr_monotonicity}
		\begin{align}
			\xx_{k}^{\T} \bb - \xx_{k}^{\T} \AA \xx_{k} &> 0, \label{eq:xTr_monotonicity1} \\
			\xx_{k}^{\T} \rr_{k} > \xx_{k-1}^{\T} \rr_{k} &\geq 0, \label{eq:xTr_monotonicity2}\\ \xx_{k}^{\T} \rr_{k-1} > \xx_{k}^{\T} \rr_{k} &> 0, \label{eq:xTr_monotonicity3}
		\end{align}
	\end{subequations}
	and the equality in \cref{eq:xTr_monotonicity2} holds only at $ k = 1 $.
\end{theorem}
\begin{proof}
	We prove \cref{eq:xTr_monotonicity1,eq:xTr_monotonicity2} by induction. For $ k = 1 $, we have 
	\begin{align*}
		\xx_{1}^{\T} \rr_{1} &= (\xx_{0} + \tau_{1} \dd_{1})^{\T} \rr_{1} = \xx_{0}^{\T} \rr_{1} + \tau_{1} \dd_{1}^{\T} \rr_{1} > \xx_{0}^{\T} \rr_{1} = 0,
	\end{align*}    
	where the inequality follows from \cref{lemma:tau_dTr}. Now suppose \cref{eq:xTr_monotonicity1,eq:xTr_monotonicity2} holds for some $ k-1 $ with $ k > 1 $. Using the fact that $ \xx_{k-1} \in \mathcal{K}_{k-1}(\AA,\bb) \perp \vv_{k+1} $ as well as \cref{lemma:tau_dTr}, we get 
	\begin{align*}
		\xx_{k}^{\T} \rr_{k} = (\xx_{k-1} + \tau_{k} \dd_{k})^{\T} \rr_{k} = \xx_{k-1}^{\T} \rr_{k} + \tau_{k} \dd_{k}^{\T} \rr_{k} > \xx_{k-1}^{\T} \rr_{k} = \xx_{k-1}^{\T} \left( s_{k}^2 \rr_{k-1} - \phi_{k} c_{k} \vv_{k+1} \right) = s_{k}^2 \xx_{k-1}^{\T}  \rr_{k-1} > 0,
	\end{align*}
	where the last strict inequality follows from the inductive hypothesis and the fact that for any $1 <  k < g $ we must have  $ \beta_{k+1} \neq 0$, which in turn implies $ s_{k} \neq 0 $. This gives us both \cref{eq:xTr_monotonicity1,eq:xTr_monotonicity2}.
	
	Finally, by \cref{eq:s_c},
	\begin{align*}
		\xx_{k}^{\T} \rr_{k} = \xx_{k}^{\T} (s_{k}^2 \rr_{k-1} - \phi_{k} c_{k} \vv_{k+1}) = s_{k}^2 \xx_{k}^{\T} \rr_{k-1} < \xx_{k}^{\T} \rr_{k-1},
	\end{align*}
	which proves \cref{eq:xTr_monotonicity3}. 
\end{proof}

\vspace{1mm}
\begin{remark}
	\label{rem:qaud}
	Suppose $ \AA \succ \zero $. For the iterates of CG, we have $\dotprod{\xx_{k}, \bb} \geq \dotprod{\xx_{k}, \AA \xx_{k}}/2$. However, \cref{eq:xTr_monotonicity1} implies that the iterates of MINRES satisfy $\dotprod{\xx_{k}, \bb} \geq \dotprod{\xx_{k}, \AA \xx_{k}}$ (without the factor $1/2$), which is a stronger property. 
	In the context of solving $ \min_{\ww} f(\ww) $, this can have significant consequences for optimization algorithms that seek a descent direction at every iteration. Indeed, suppose $\AA = \nabla^{2} f(\ww)$ and $\bb = - \nabla f(\ww)$. From \cref{eq:xTr_monotonicity1}, it follows that the angle between the iterates of MINRES and the gradient of $ f $ is expected to be more negative than that given by the iterates of CG. In this light, loosely speaking, MINRES can produce better directions in that the amount of descent can be greater than the equivalent directions obtained from CG. 
\end{remark}

\vspace{1mm}
\begin{remark}
	\label{rem:monoton01}
	The conclusions of \cref{thm:xTr_monotonicity} hold only for when $ k < g $. For the last iteration $ k=g $, we need to consider two separate cases. 
	\begin{enumerate}[label = {\bfseries (\roman*)}]
		\item When $ \bb \in \range(\AA) $, if the NPC condition \cref{eq:c_gama_nc} has not been detected for all $ g $ iterations, since $ \rr_{g} = \zero $, we get
		\begin{align*}
			\xx_{g}^{\T} \rr_{g-1} > \xx_{g}^{\T} \rr_{g} = \xx_{g-1}^{\T} \rr_{g} = 0.
		\end{align*}
		\item When $ \bb \notin \range(\AA) $, if the NPC condition \cref{eq:c_gama_nc} has not been detected for all $ g-1 $ iterations, it is guaranteed to be detected at the $ g\th $ iteration (\cref{lemma:Lanczos}). In this case, we have $ \gamma^{(2)}_{g} = 0 $, which implies $ \tau_{g} = 0 $, and as a result $ \xx_{g} = \xx_{g-1} $ and $ \rr_{g} = \rr_{g-1} $. Now, it can be seen from the proof of \cref{thm:xTr_monotonicity} that
		\begin{align*}
			\xx_{g}^{\T} \bb - \xx_{g}^{\T} \AA \xx_{g} = \xx_{g}^{\T} \rr_{g} = \xx_{g-1}^{\T} \rr_{g} =
			\xx_{g}^{\T} \rr_{g-1} = \xx_{g-1}^{\T} \rr_{g-1} > 0.
		\end{align*}
	\end{enumerate}
\end{remark}

\cref{thm:A_norm} provides a certain set of monotonicity results for MINRES, which for the special case of a positive definite matrix $ \AA $, mimic those of CG.

\begin{theorem}
	\label{thm:A_norm} 
	Let $ g $ be the grade of $ \bb $ with respect to $ \AA $ as in \cref{def:grade_b}. As long as the NPC condition \cref{eq:c_gama_nc} has not been detected for $1 \leq k < g $, the following monotonicity results hold.
	\begin{enumerate}[label = {\bfseries (\alph*)}]
		\item \label{thm:A_norm:energy_norm} If $ \AA \xx^{\star} = \bb $ for some $ \xx^{\star} \in \real^{d} $, i.e., $ \bb \in \range(\AA) $, then $ \| \xx^{\star} - \xx_{k} \|_{\AA} $ decreases strictly monotonically with $ k $.
		\item \label{thm:A_norm:energy_norm_omega} If $ \AA \xx^{\star} = \bb $ for some $ \xx^{\star} \in \real^{d} $, i.e., $ \bb \in \range(\AA) $, then $ \| \xx^{\star} - (\xx_{k-1} + \omega \tau_{k} \dd_{k})  \|_{\AA} $ decreases strictly monotonically over $ \omega \in (-\infty, 1] $. 
		\item $ m(\xx_{k}) \triangleq \dotprod{\xx_{k}, \AA \xx_{k}}/2 - \dotprod{\bb,\xx_{k}} $ decreases strictly monotonically with $ k $.
		\item $ m(\xx_{k-1} + \omega \tau_{k} \dd_{k}) $ decreases strictly monotonically over $ \omega \in (-\infty, 1] $.
		\item $ \| \xx_{k} \| $ increases strictly monotonically with $ k $.
		\item $ \| \xx_{k-1} + \omega \tau_{k} \dd_{k} \| $ increases strictly monotonically over $ \omega \in [0, \infty) $.
		\item $ \xx_{k}^{\T} \bb $ increases strictly monotonically with $ k $.
	\end{enumerate}
\end{theorem}

\begin{proof}
	\hfill
	\begin{enumerate}[label = {\bfseries (\alph*)}]
		\item We have
		\begin{align*}
			\vnorm{\xx^{\star} - \xx_{k}}_{\AA}^2 - \vnorm{\xx^{\star} - \xx_{k-1}}_{\AA}^2 &= \left( \xx^{\star} - \xx_{k} \right)^{\T} \AA \left( \xx^{\star} - \xx_{k} \right) - \left( \xx^{\star} - \xx_{k-1} \right)^{\T} \AA \left( \xx^{\star} - \xx_{k-1} \right) \\
			&= - 2 \xx_{k}^{\T} \bb + \xx_{k}^{\T} \AA \xx_{k} + 2 \xx_{k-1}^{\T} \bb - \xx_{k-1}^{\T} \AA \xx_{k-1} \\
			&= - 2 \left( \xx_{k-1} + \tau_{k} \dd_{k} \right)^{\T} \bb + \left( \xx_{k-1} + \tau_{k} \dd_{k} \right)^{\T} \AA \left( \xx_{k-1} + \tau_{k} \dd_{k} \right) \\
			& \quad + 2 \xx_{k-1}^{\T} \bb - \xx_{k-1}^{\T} \AA \xx_{k-1} \\
			&= - 2 \tau_{k} \dd_{k}^{\T} \bb + \tau_{k} \dd_{k}^{\T} \AA \xx_{k} +  \tau_{k} \dd_{k}^{\T} \AA \xx_{k-1} = - \tau_{k} \dd_{k}^{\T} \rr_{k} - \tau_{k} \dd_{k}^{\T} \rr_{k-1} < 0,
		\end{align*}
		where the last inequality comes from \cref{lemma:tau_dTr}.
		\item Let $ \omega_{1} < \omega_{2} \leq 1 $. We have
		\begin{align*}
			&\quad \vnorm{\xx^{\star} - (\xx_{k-1} + \omega_{2} \tau_{k} \dd_{k})}_{\AA}^2 - \vnorm{\xx^{\star} - (\xx_{k-1} + \omega_{1} \tau_{k} \dd_{k})}_{\AA}^2 \\
			&= - 2 (\omega_{2} - \omega_{1}) \tau_{k} \dd_{k}^{\T} \bb + 2 (\omega_{2} - \omega_{1}) \tau_{k} \dd_{k} \AA \xx_{k-1} + (\omega_{2}^2 - \omega_{1}^2) \tau_{k}^2 \dd_{k}^{\T} \AA \dd_{k} \\
			&< - 2 (\omega_{2} - \omega_{1}) \tau_{k} \dd_{k}^{\T} \bb + 2 (\omega_{2} - \omega_{1}) \tau_{k} \dd_{k} \AA \xx_{k-1} + 2 (\omega_{2} - \omega_{1}) \tau_{k}^2 \dd_{k}^{\T} \AA \dd_{k} \\
			&= - 2 (\omega_{2} - \omega_{1}) \tau_{k} \dd_{k}^{\T} \bb + 2 (\omega_{2} - \omega_{1}) \tau_{k} \dd_{k} \AA \xx_{k} = - 2 (\omega_{2} - \omega_{1}) \tau_{k} \dd_{k}^{\T} \rr_{k} < 0,
		\end{align*}
		where the first inequality follows from the facts that $ \omega_{2} + \omega_{1} < 2 $, $ (\omega_{2}^2 - \omega_{1}^2) = (\omega_{2} + \omega_{1}) (\omega_{2} - \omega_{1})$, and $ \dd_{k}^{\T} \AA \dd_{k} > 0 $, which is due to $ \dd_{k} \in \mathcal{K}_{k}(\AA, \bb) $ and $ \TT_{k} \succ \zero $ by \cref{thm:determinant_T}, and the last inequality is given using \cref{lemma:tau_dTr}. 
		\item Similar to the proof of \labelcref{thm:A_norm:energy_norm}, we have
		\begin{align*}
			m(\xx_{k}) - m(\xx_{k-1}) &= - \xx_{k}^{\T} \bb + \hf \xx_{k}^{\T} \AA \xx_{k} + \xx_{k-1}^{\T} \bb - \hf \xx_{k-1}^{\T} \AA \xx_{k-1} \\
			&= - \left( \xx_{k-1} + \tau_{k} \dd_{k} \right)^{\T} \bb + \hf \left( \xx_{k-1} + \tau_{k} \dd_{k} \right)^{\T} \AA \left( \xx_{k-1} + \tau_{k} \dd_{k} \right) + \xx_{k-1}^{\T} \bb - \hf \xx_{k-1}^{\T} \AA \xx_{k-1} \\
			&= - \tau_{k} \dd_{k}^{\T} \bb + \hf \tau_{k} \dd_{k}^{\T} \AA \xx_{k} +  \hf \tau_{k} \dd_{k}^{\T} \AA \xx_{k-1} = - \hf \tau_{k} \dd_{k}^{\T} \rr_{k} - \hf \tau_{k} \dd_{k}^{\T} \rr_{k-1} < 0,
		\end{align*}
		where again the last inequality follows from \cref{lemma:tau_dTr}.
		\item Similar to the proof of \labelcref{thm:A_norm:energy_norm_omega}, let $ \omega_{1} < \omega_{2} \leq 1 $ and we have
		\begin{align*}
			& \quad m(\xx_{k-1} + \omega_{2} \tau_{k} \dd_{k}) - m(\xx_{k-1} + \omega_{1} \tau_{k} \dd_{k}) \\
			&= - (\xx_{k-1} + \omega_{2} \tau_{k} \dd_{k})^{\T} \bb + \hf (\xx_{k-1} + \omega_{2} \tau_{k} \dd_{k})^{\T} \AA (\xx_{k-1} + \omega_{2} \tau_{k} \dd_{k}) \\
			&\quad + (\xx_{k-1} + \omega_{1} \tau_{k} \dd_{k})^{\T} \bb - \hf (\xx_{k-1} + \omega_{1} \tau_{k} \dd_{k})^{\T} \AA (\xx_{k-1} + \omega_{1} \tau_{k} \dd_{k}) \\
			&= - (\omega_{2} - \omega_{1}) \tau_{k} \dd_{k}^{\T} \bb + (\omega_{2} - \omega_{1}) \tau_{k} \dd_{k} \AA \xx_{k-1} + \hf (\omega_{2}^2 - \omega_{1}^2) \tau_{k}^2 \dd_{k}^{\T} \AA \dd_{k} \\
			&< - (\omega_{2} - \omega_{1}) \tau_{k} \dd_{k}^{\T} \bb + (\omega_{2} - \omega_{1}) \tau_{k} \dd_{k} \AA \xx_{k-1} + (\omega_{2} - \omega_{1}) \tau_{k}^2 \dd_{k}^{\T} \AA \dd_{k} \\
			&= - (\omega_{2} - \omega_{1}) \tau_{k} \dd_{k}^{\T} \bb + (\omega_{2} - \omega_{1}) \tau_{k} \dd_{k} \AA \xx_{k} = - (\omega_{2} - \omega_{1}) \tau_{k} \dd_{k}^{\T} \rr_{k} < 0,
		\end{align*}    
		where the last inequality comes from \cref{lemma:tau_dTr}.
		
		\item we have
		\begin{align*}
			\vnorm{\xx_{k}}^2 - \vnorm{\xx_{k-1}}^2 &= \left( \xx_{k-1} + \tau_{k} \dd_{k} \right)^{\T} \left( \xx_{k-1} + \tau_{k} \dd_{k} \right) - \vnorm{\xx_{k-1}}^2 \\
			&= 2 \tau_{k} \dd_{k}^{\T} \xx_{k-1} + \tau_{k}^2 \dd_{k}^{\T} \dd_{k} = \tau_{k} \dd_{k}^{\T} \xx_{k-1} + \tau_{k} \dd_{k}^{\T} \xx_{k} > 0,
		\end{align*}
		where the last inequality comes from \cref{lemma:tau_dTx}.
		
		\item Let $ \omega_{1} < \omega_{2} $. We have
		\begin{align*}
			\vnorm{\xx_{k-1} + \omega_{2} \tau_{k} \dd_{k}}^2 - \vnorm{\xx_{k-1} + \omega_{1} \tau_{k} \dd_{k}}^2 &= 2 (\omega_{2} - \omega_{1}) \tau_{k} \dd_{k}^{\T} \xx_{k-1} + (\omega_{2}^2 - \omega_{1}^2) \tau_{k}^2 \vnorm{\dd_{k}}^2 > 0,
		\end{align*}
		where the last inequality follows from \cref{lemma:tau_dTx} and noting that $ \omega_{1} + \omega_{2} > 0 $. 
		
		\item By \cref{lemma:tau_dTb}, we simply get the result as 
		\begin{align*}
			\xx_{k}^{\T} \bb - \xx_{k-1}^{\T} \bb = \tau_{k} \dd_{k}^{\T} \bb > 0.
		\end{align*}
	\end{enumerate}
\end{proof}

\begin{remark}
	\label{rem:monoton02}
	The conclusions of \cref{thm:A_norm} hold only when $ k < g $. For the last iteration $ k=g $, we again need to consider two separate cases. 
	\begin{enumerate}[label = {\bfseries (\roman*)}]
		\item When $ \bb \in \range(\AA) $, we always have $ \gamma^{(2)}_{g} \neq 0 $ (see \cref{rem:gamma_2}). In this case, if the NPC condition \cref{eq:c_gama_nc} has not been detected for all $ g $ iterations, the conclusions of \cref{thm:A_norm} continue to hold for $ k = g $.
		\item When $ \bb \notin \range(\AA) $, as discussed in \cref{rem:monoton01}, we always have $ \gamma^{(2)}_{g} = 0 $, $ \tau_{g} = 0 $, and $ \xx_{g} = \xx_{g-1} $. Also, if the NPC condition \cref{eq:c_gama_nc} has not been detected for all $ g-1 $ iterations, it is guaranteed to be detected at the $ g\th $ iteration. This implies that at the very last iteration, we have 
		\begin{align*}
			m(\xx_{g}) &= m(\xx_{g-1}), \quad \vnorm{\xx_{g}} = \vnorm{\xx_{g-1}}, \quad \text{and} \quad \dotprod{\xx_{g}, \bb} = \dotprod{\xx_{g-1}, \bb}.
		\end{align*}
	\end{enumerate}
\end{remark}


\section{Numerical experiments}
\label{sec:exp}

We now give several numerical experiments to not only verify our main theoretical results (\cref{sec:exp:theory}), but also to showcase the advantages of using the NPC direction detected as part of the MINRES iterations for optimization algorithms (\cref{sec:exp:NewtonMR}). 

\subsection{NPC detection and monotonicity properties}
\label{sec:exp:theory}
In this section, we numerically verify the results from \cref{thm:determinant_T,thm:A_norm,thm:xTr_monotonicity}
To do this, we consider $ d = 20 $ and generate symmetric matrices $ \AA $, $ \BB $ and $ \CC $ as three random realizations of the Gaussian orthogonal ensemble. We then manually change the top $ 19 $ eigenvalues to be logarithmically spaced points in the interval $ [1, 10^3] $. The last eigenvalues of $ \AA $ and $ \BB $ are set to $0$ and $-1$, respectively, while the last two eigenvalues of $ \CC $ are chosen to be $ -1 $ and $ -10 $. As a result, the first $ 18 $ positive eigenvalues of $ \CC $ are the same as those of $ \AA $ and $ \BB $. We then consider solving the least-squared problem using \cref{alg:MINRES} with $ \bb = \one $, i.e., the vector of all ones, and the underlying matrices $ \AA $, $ \BB $ and $ \CC $. \cref{fig:minres} depicts several quantities of interest across the iterations of MINRES on these three problem instances.  

In \cref{fig:minres}, small special marks, i.e., ``$ \bm{\times} $'',``$ \bm{+} $'', and ``$\bm{\star}$'', represent MINRES iterates, which are then connected to one another in sequence using dashed lines. The special marks corresponding to those iterations where the NPC condition \cref{eq:c_gama_nc} is detected are enlarged to distinguish them from other iterates. 
Let us consider all the iterates before the NPC condition \cref{eq:c_gama_nc} is detected for the first time. For all these iterations, we see that the smallest eigenvalue of $ \TT_{k} $ remains strictly positive, i.e., $ \TT_{k} \succ \zero $, which verifies \cref{thm:determinant_T}. 
Throughout these iterations, the quantity $ \dotprod{\xx_{k}, \rr_{k}} $ remains positive, which validates the result of \cref{thm:xTr_monotonicity}.
Similarly, for all these iterates, as predicted by \cref{thm:A_norm}, the quadratic $ m(\xx_{k}) \triangleq \dotprod{\xx_{k}, \AA \xx_{k}}/2 - \dotprod{\bb,\xx_{k}} $ is monotonically decreasing while the quantities $ \vnorm{\xx_{k}} $ and $ \dotprod{\xx_{k},\bb} $ are monotonically increasing.  
In sharp contrast, as soon as the NPC condition \cref{eq:c_gama_nc} is detected for the first time, the above monotonicity properties are no longer guaranteed and can be violated as of that iteration. 
Note that, since the construction of the matrix $ \AA $ ensures that $ \AA \succeq \zero $ and $ \bb \notin \range(\AA) $, the residual of the corresponding system will not vanish, and as anticipated by \cref{lemma:Lanczos}, the NPC condition \cref{eq:c_gama_nc} is detected at the very last iteration.

\begin{figure}[!htbp]
	\centering
	\hspace*{-1.5cm}
	\includegraphics[width=1\textwidth]{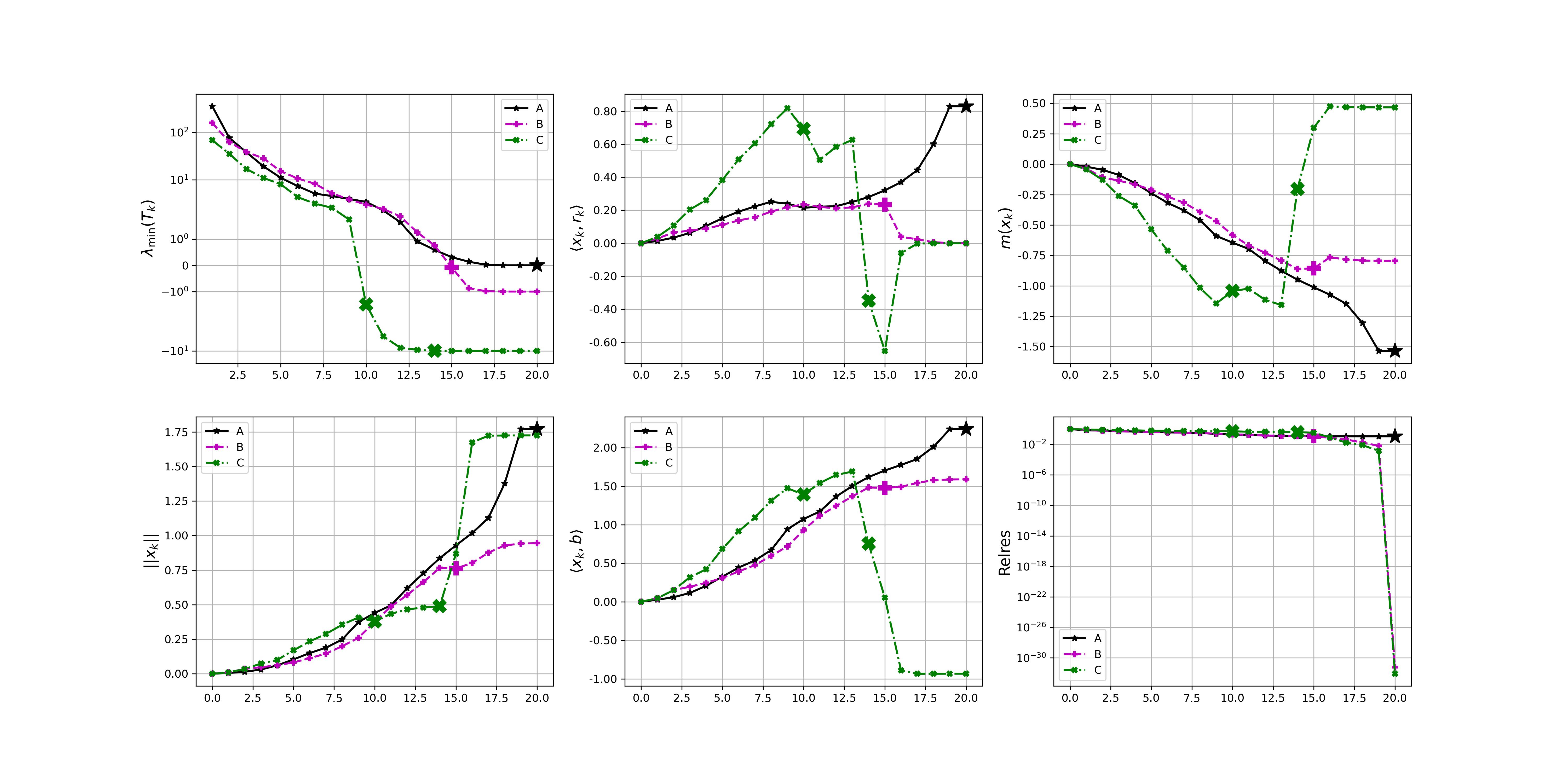}
	\caption{Several relevant quantities, namely $ \lambda_{\min}(\TT_k) $, $ \dotprod{\xx_{k}, \rr_{k}} $, $ m(\xx_{k}) \triangleq \dotprod{\xx_{k}, \AA \xx_{k}}/2 - \dotprod{\bb,\xx_{k}} $, $ \vnorm{\xxk} $, $ \dotprod{\xx_{k}, \bb} $,  and $ \vnorm{\rrk}/\vnorm{\bb} $, across iterations of MINRES using $ d = 20 $, $ \bb = \one $, and with matrices $ \AA $, $ \BB $ and $ \CC $ as constructed in \cref{sec:exp:theory}. In all the plots, the x-axis represents the iteration counter and large special marks on each plot highlight the iterations where the NPC condition \cref{eq:c_gama_nc} is detected. \label{fig:minres}}
\end{figure}

\subsection{NPC direction and optimization algorithms}
\label{sec:exp:NewtonMR}
We now move onto to present preliminary experiments to showcase advantages of using a NPC direction provided by MINRES for optimization algorithms. We do this by consider an unconstrained minimization problem involving a twice continuously differentiable function $ f:\real^{d} \to \real $. We focus on our prior work \cite{roosta2018newton,liu2021convergence}, henceforth referred to as \texttt{Newton-MR-grad}, as a variant of inexact Newton Method in which the least-squares subproblems are approximately solved using MINRES. By construction, \texttt{Newton-MR-grad} solves $\min_{\ww} \vnorm{\nabla f(\ww)}^{2} $ as a proxy for solving $\min_{\ww} f(\ww)$, and hence can be applied for unconstrained optimization of a class
of non-convex problems known as invex \cite{mishra2008invexity}. However, beyond invex problems, \texttt{Newton-MR-grad} does not provide any guarantees on the quality of the solutions and, therefore, it may terminate near a saddle point or even a local maximum. This can be remedied by employing a NPC direction arising as part of the MINRES iterations. 

For this, we consider a new variant, referred to as \texttt{Newton-MR}, which differs from \texttt{Newton-MR-grad} by two simple modifications. First, unlike \texttt{Newton-MR-grad}, the iterations of MINRES for \texttt{Newton-MR} are terminated as soon as a NPC condition is detected and the corresponding NPC direction is return to be used for line-search purposes. Second, the Armijo-type line search is performed to obtain a step-size to reduce $ f(\ww) $, as opposed to $ \| \nabla f(\ww) \|^2 $ for \texttt{Newton-MR-grad}. For both methods, the initial trial step-size for line-search is set to one. As a consequence of this construction, starting from the same initial point, $ \ww_{0} $, both algorithms take identical steps until either a NPC direction is detected within MINRES or the step-sizes returned as part of the corresponding line-search procedures are different.

Following \cite{xuNonconvexEmpirical2017}, we consider a binary classification problem using a (regularized) nonlinear least-square objective of the form
\begin{align}
	\label{eq:NLS}  
	f(\ww) = \frac{1}{n} \sum_{i=1}^{n}  \left(\frac{1}{1+e^{-\dotprod{\aa_i, \ww}}} - b_i \right)^2 + \psi(\ww),
\end{align}
where $ \psi(\ww) $ is some regularization function, and the dataset $ \{\aa_i, b_i\}_{i=1}^n $ consists of feature vectors $ \aa_i \in \mathbb{R}^d, \; i=1,\ldots,n $ and the corresponding labels $ b_i \in \{0, 1\}, \; i=1,\ldots,n $. We consider three different regularization strategies, namely, a typical convex $ \ell_{2} $ regularization $ \psi(\ww) = 0.5 \| \ww \|^2 $, a nonconvex regularization $ \psi(\ww) = 0.01 \sum_{i=1}^{d} w_i^2 / (1+w_i^2) $, and also the case where there is no regularization, $ \psi(\ww) = 0 $. We have used \texttt{CIFAR10} dataset \cite{krizhevsky2009learning}, which contains $60,000$ color images of size $32 \times 32 $ in $10$ classes. For our binary classification setting, we have relabeled the odd and even classes, respectively, as $ 0 $ and $1 $. 

We terminate the optimization algorithms once the Euclidean norm of the gradient of $ f $ reaches below $ 10^{-10} $. The MINRES iterations are terminated if a relative residual tolerance of $0.01$ is reached (\texttt{Newton-MR} and \texttt{Newton-MR-grad}) or if a NPC direction is detected (\texttt{Newton-MR}).  To measure performance, we use the same metric as that considered in \cite{roosta2018newton}. Specifically, we consider total number of oracle calls of the function, gradient and Hessian-vector product as a complexity measure. In this light, respectively, total number of oracle calls for every iteration of \texttt{Newton-MR-grad} and \texttt{Newton-MR} will be $ 2 N_s + 2 N_{l} + 2$ and  $ 2 N_s + N_{l} + 2$, where $N_s$ and $N_l$ denote the total number of iterations for MINRES and the line search.

\begin{figure}[!htbp]
	\centering
	\hspace*{-1cm}
	\subfigure[Convex regularization]
	{\includegraphics[scale=0.3]{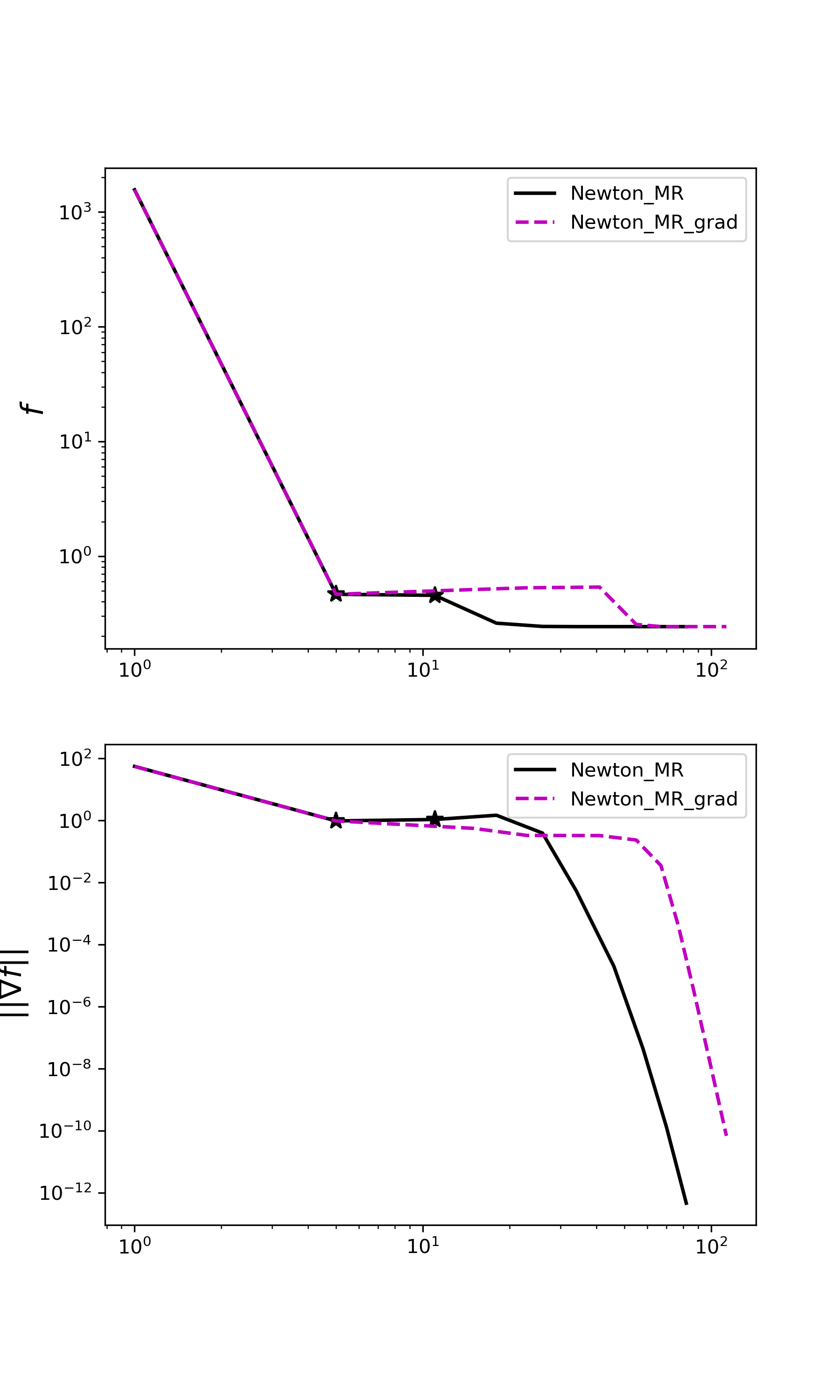}}
	\subfigure[No regularization]
	{\includegraphics[scale=0.3]{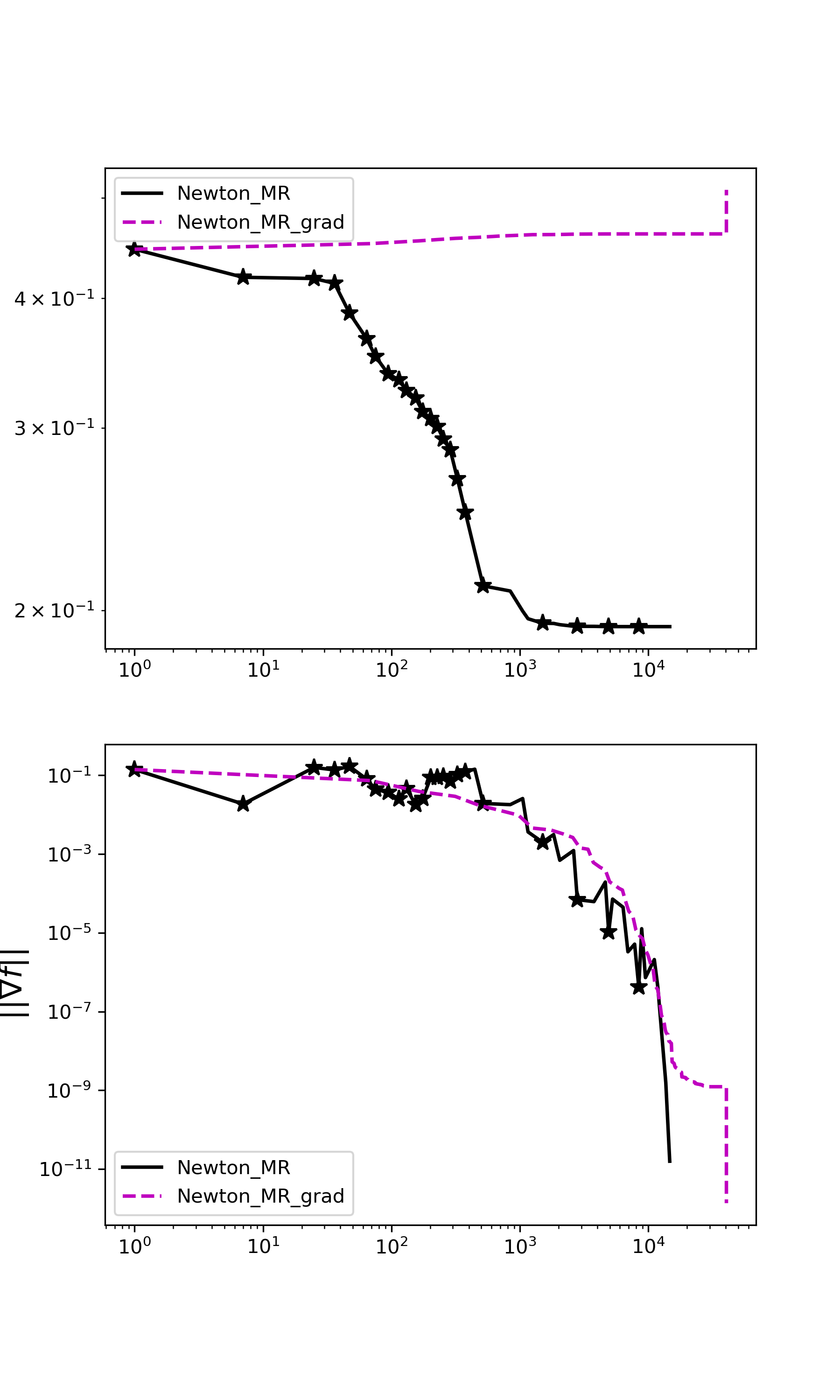}}
	\subfigure[Nonconvex regularization]
	{\includegraphics[scale=0.3]{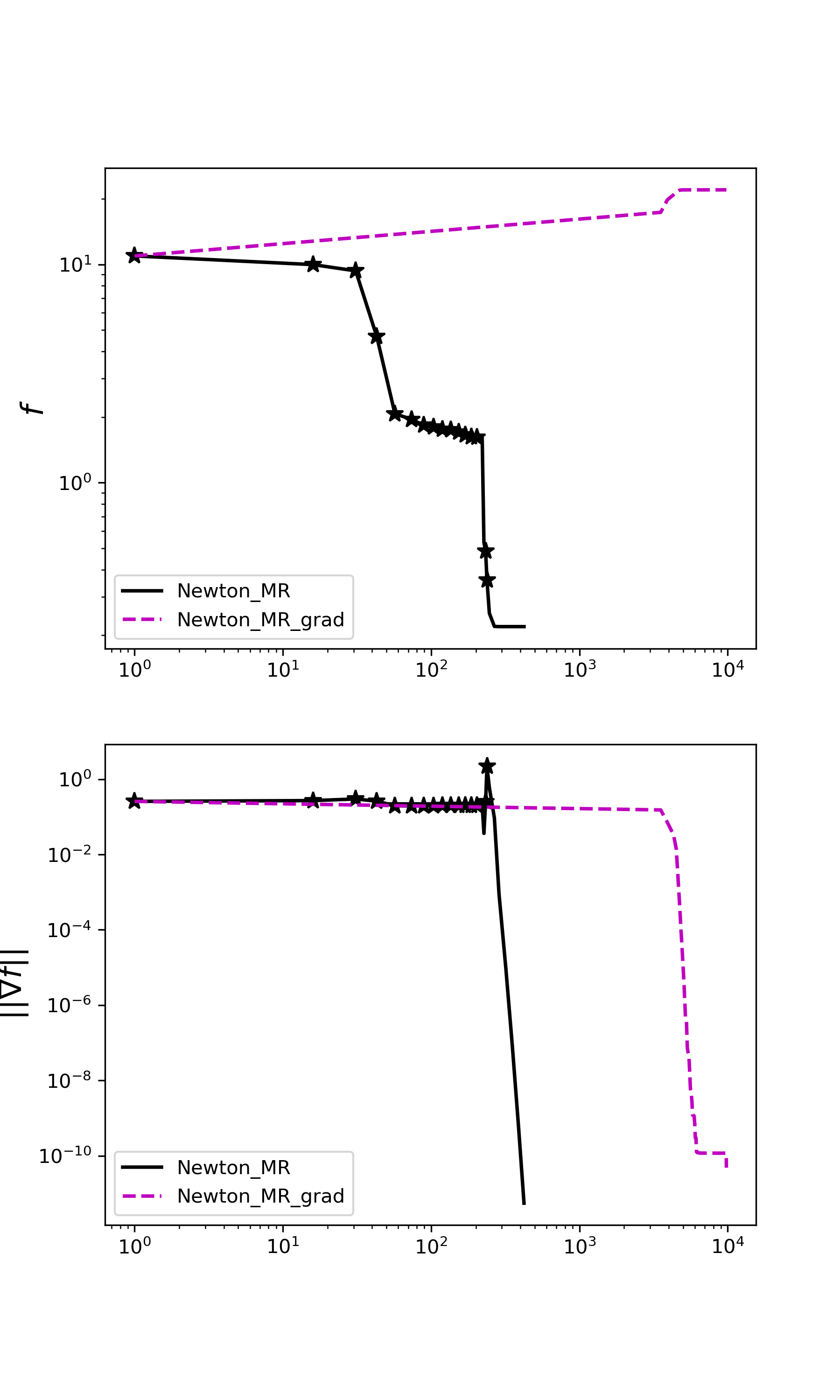}}
	\caption{The performance of \texttt{Newton-MR} v.s.\ \texttt{Newton-MR-grad} on the (regularized) non-linear least squares objective \cref{eq:NLS} using three different regularization strategies. The $x$-axis represents the total number of oracle calls. In the plots corresponding to \texttt{Newton-MR}, the special marks, denoted by ``$\bm{\star}$'', represent iterations where MINRES detects a NPC direction, which is then used as an update direction within the line-search procedure. \label{fig:NewtonMR}}
\end{figure}

\cref{fig:NewtonMR} shows the performance of \texttt{Newton-MR-grad} and \texttt{Newton-MR} across three regularization functions. In the plots corresponding to \texttt{Newton-MR}, the special marks, denoted by ``$\bm{\star}$'', represent iterations where MINRES detects a NPC direction, which is then used as an update direction within the line-search procedure for \texttt{Newton-MR}. When the regularization is convex, both methods perform similarly. Despite the fact that this problem is not invex, the presence of convex regularization provides a favorable optimization landscape for \texttt{Newton-MR-grad}, and as a result, both methods are able to find a reasonable solution with a comparative amount of work.  However, beyond the convex regularization, the benefits of using NPC direction become quite clear.  Indeed, the nonconvex regularization and removing regularization term altogether, both, lead to highly non-convex optimization landscapes for which \texttt{Newton-MR-grad} clearly converges to undesirable local maxima. However, \texttt{Newton-MR}, which can employ NPC directions detected as part of MINRES iterations, gives monotonic decrease in objective value and converges to satisfactory local minima. Establishing the theoretical guarantees of our prototype algorithm, \texttt{Newton-MR}, is the subject of a future work.

\section{Conclusions}
\label{sec:conclusion}
For solving linear least-squares problems involving a real symmetric but potentially indefinite and/or singular  matrix $ \AA $, we considered the celebrated minimal residual (MINRES) method of Paige and Saunders \cite{paige1975solution}. We showed that MINRES comes equipped with an inherent ability to detect directions of nonpositive curvature. Such a direction can be detected by monitoring a certain readily available condition within the MINRES iterations (NPC Condition \cref{eq:c_gama_nc}). We showed that whenever such a condition holds, the residual of the previous iteration is one such direction. As a consequence, MINRES has a built-in mechanism to provide a certificate for the positive semi-definiteness (or lack thereof) of $ \AA $. 
We then established several monotonicity properties, that mimic those of CG but are applicable for any real symmetric matrix. We also numerically verified our main results and showcased the advantages of using the NPC directions arising as part of MINRES iterations for optimization algorithms. As anticipated in \cite{dahito2019conjugate}, it is hoped that these properties will allow MINRES to be considered as a potentially superior alternative to CG for all Newton-type nonconvex optimization algorithms that employ CG as their subproblem solver. 

\appendix
\section{Technical Lemmas}
\label{sec:appendix}
To establish the main results of this paper, we need a few technical lemmas, which are all gathered here. 

The following \cref{lemma:basic,lemma:tvr,lemma:tau_d} are used in the proof of \cref{thm:determinants,thm:xTr_monotonicity}.

\begin{lemma}
	\label{lemma:basic}
	Let $ g $ be the grade of $ \bb $ with respect to $ \AA $ as in \cref{def:grade_b}. As long as the NPC condition \cref{eq:c_gama_nc} has not been detected for $ 1 \leq k \leq g $, we have
	\begin{subequations}
		\label{eq:basic}
		\begin{align}
			\alpha_{k} &> 0, \quad \beta_{k} > 0   \label{eq:alpha_beta} \\
			1 > s_{k} &\geq 0, \quad 1 \geq \abs{c_t} > 0,   \label{eq:s_c} \\
			(-1)^{k-i} c_{i} \gamma_{k}^{(1)} &> 0, \quad (-1)^{k-i} c_{i} \tau_{k} > 0, \quad (-1)^{k-i} c_{i} c_{k} > 0, \qquad 0 \leq i \leq k, \label{eq:c_gama_tau_c} \\
			\gamma_{k}^{(2)} &> 0, \quad \delta_{k}^{(2)} > 0, \quad \epsilon_{k} > 0, \label{eq:R_entries} \\
			& \hspace{-15mm} \Span \{ \rr_{0}, \ldots, \rr_{k-1}\} = \mathcal{K}_{k} (\AA, \bb).  \label{eq:krylov_rk}
		\end{align} 
	\end{subequations}
\end{lemma}
\begin{proof}
	We first note that if the NPC condition \cref{eq:c_gama_nc} does not hold for any $ k \leq g $, then we must have $ c_{k-1} \neq 0 $, $ \phi_{k-1} \neq 0 $, $ \gamma_{k}^{(1)} \neq 0 $, and consequently $ \gamma_{k}^{(2)} \neq 0 $ by \cref{eq:c_s_r2}. Also, from \cref{eq:rAr,thm:determinant_T}, as long as the NPC condition \cref{eq:c_gama_nc} has not been detected for $ k \leq g $, then we must have $ \rr_{k-1}^{\T} \AA \rr_{k-1} > 0 $ and $ \TT_{k} \succ \zero $.
	
	By construction in \cref{alg:MINRES}, we always have $ \beta_{k} > 0, \; k \leq g $, and from $ \TT_{k} \succ \zero $ it follows that $ \alpha_{k} > 0, \; k \leq g  $. This gives \cref{eq:alpha_beta}.
	
	From \cref{eq:c_s_r2} and the fact that $ c_{k-1} \neq 0 $, we get $ 0 \leq s_{k-1} < 1 $ and $ 0 < |c_{k-1}| \leq 1 $, which gives \cref{eq:s_c}.

	By \cref{eq:c_s_r2} and the construction of \cref{alg:MINRES}, we know that, if $ \gamma_{i}^{(1)} \neq 0 $, then $ c_{i}, \tau_{i}$ and $\gamma_{i}^{(1)} $ all have the same sign. Using the assumption that the NPC condition \cref{eq:c_gama_nc} does not hold, coupled with the fact that $ c_{0} = -1 $, we get 
	\begin{align*}
		\left\{
		\begin{array}{ll}
			c_{i} > 0, \tau_{i} > 0, \gamma_{i}^{(1)} > 0, \quad \text{$ i $ is odd}, \\
			c_{i} < 0, \tau_{i} < 0, \gamma_{i}^{(1)} < 0, \quad \text{$ i $ is even},
		\end{array}
		\right.
	\end{align*}
	which gives us \cref{eq:c_gama_tau_c}.      
	Again, from \cref{alg:MINRES,eq:c_gama_tau_c}, we have that 
	\begin{align*}
		\gamma_{k}^{(2)} = \sqrt{(\gamma_{k}^{(1)})^2 + \beta_{k+1}^2} &> 0, \quad k \geq 1, \\
		\delta_{k}^{(2)} = -c_{k-2} c_{k-1} \beta_{k} + s_{k-1} \alpha_{k} &> 0, \quad k \geq 2, \\
		\epsilon_{k} = s_{k-2} \beta_{k} &> 0, \quad k \geq 3,
	\end{align*}    
	which gives \cref{eq:R_entries}. 
	
	To prove \cref{eq:krylov_rk}, we need to consider \cref{eq:rk_rec}. If $ \phi_{k-1} \neq 0 $ and  $ c_{k} \neq 0 $, it follows that $ \vv_{k} \in \Span\{\rr_{k-1},\rr_{k-2}\} $, which in turn implies 
	\begin{align*}
		\Span \{ \rr_{0}, \ldots, \rr_{k-1}\} = \Span \{ \vv_{1}, \ldots, \vv_{k}\} = \mathcal{K}_{k} (\AA, \bb).
	\end{align*}
\end{proof}

\begin{remark}
	\label{rem:gamma_2}
	Since we always have $ \beta_{g+1} = 0 $, from the construction of \cref{alg:MINRES}, it follows that $ \gamma_{g}^{(2)} \neq 0 $ is equivalent to having $ \rr_{g} = \zero $, i.e., $ \bb \in \range(\AA) $. In this light, in what follows, the condition $ \gamma_{g}^{(2)} \neq 0  $ can be thought of as implying $ \bb \in \range(\AA) $.
\end{remark}

\begin{lemma}
	\label{lemma:tvr}   
	Let $ g $ be the grade of $ \bb $ with respect to $ \AA $  as in \cref{def:grade_b}. As long as the NPC condition \cref{eq:c_gama_nc} has not been detected for $ 1 \leq k < g $, we have
	\begin{align*}
		(-1)^{i} \tau_{k} \vv_{k-i}^{\T} \rr_{k-j} > 0, \quad 0 \leq i < k, \quad 0 \leq j \leq i + 1. 
	\end{align*}
	Furthermore, if $ \gamma^{(2)}_{g} \neq 0 $ and the NPC condition \cref{eq:c_gama_nc} has not been detected for all $ g $ iterations, the conclusion continues to hold for $ k = g $ and $ 1 \leq j \leq i+1 $.
\end{lemma}
\begin{proof}
	Since $ 0 \leq i < k < g $, we have $\phi_{k-i-1} > 0$, which from $ \phi_{k-i-1} = s_{k-i-1} s_{k-i-2} \ldots s_{1} \neq 0 $ and \cref{eq:s_c} implies $ 0 < s_{\ell} < 1 $ for $ 1 \leq \ell \leq k-i-1 $. Also from \cref{eq:krylov_rk}, we have $ \Span \{ \rr_{0}, \ldots, \rr_{\ell-1}\} = \mathcal{K}_{\ell} (\AA, \bb) = \Span \{ \vv_{1}, \ldots, \vv_{\ell}\} $. First, consider the case where $ j = i+1 < k $. From \cref{alg:MINRES} and \cref{eq:c_gama_tau_c}, we have
	\begin{align*}
		(-1)^{i} \tau_{k} \vv_{k-i}^{\T} \rr_{k-i-1} = (-1)^{i} \tau_{k} \vv_{k-i}^{\T} \left( s_{k-i-1}^2 \rr_{k-i-2} - \phi_{k-i-1} c_{k-i-1} \vv_{k-i} \right) =  (-1)^{i+1} \tau_{k} c_{k-i-1} \phi_{k-i-1} > 0.
	\end{align*}
	Now, suppose $ j < i+1 < k $. Again, from \cref{alg:MINRES} and \cref{eq:c_gama_tau_c}, it follows that
	\begin{align*}
		(-1)^{i} \tau_{k} \vv_{k-i}^{\T} \rr_{k-j}  &= (-1)^{i} \tau_{k} \vv_{k-i}^{\T} \left( s_{k-j}^2 \rr_{k-j-1} - \phi_{k-j} c_{k-j} \vv_{k-j+1} \right) = (-1)^{i} \tau_{k} s_{k-j}^2 \vv_{k-i}^{\T} \rr_{k-j-1}\\
		& \quad \vdots \\
		&= (-1)^{i} \tau_{k} s_{k-j}^2 s_{k-j-1}^2 \ldots s_{k-i}^2 \vv_{k-i}^{\T} \rr_{k-i-1}\\
		&= (-1)^{i} \tau_{k} s_{k-j}^2 s_{k-j-1}^2 \ldots s_{k-i}^2 \vv_{k-i}^{\T} \left( s_{k-i-1}^2 \rr_{k-i-2} - \phi_{k-i-1} c_{k-i-1} \vv_{k-i} \right) \\
		&= (-1)^{i+1} \tau_{k} c_{k-i-1} \phi_{k-i-1}  \prod_{\ell=j}^{i} s_{k-\ell}^2 > 0.
	\end{align*}
	Now, let's suppose $ j = i + 1 = k $. Similarly to above, we have
	\begin{align*}
		(-1)^{k-1} \tau_{k} \vv_{1}^{\T} \rr_{0}  = (-1)^{k-1} \tau_{k} \vnorm{\bb} = (-1)^{k+1} \tau_{k} \vnorm{\bb} = (-1)^{k} c_{0} \tau_{k} \vnorm{\bb} > 0,
	\end{align*}
	where we have again used \cref{eq:c_gama_tau_c} and the fact that $ c_{0} = -1 $. Similarly, for $ j < i + 1 = k $, we have
	\begin{align*}
		(-1)^{k-1} \tau_{k} \vv_{1}^{\T} \rr_{k-j}  &= (-1)^{k-1} \tau_{k} s_{k-j}^2 s_{k-j-1}^2 \ldots s_{1}^2 \vv_{1}^{\T} \rr_{0} = (-1)^{k-1} \tau_{k} s_{k-j}^2 s_{k-j-1}^2 \ldots s_{1}^2 \vnorm{\bb} > 0.
	\end{align*}
\end{proof}

\begin{lemma}
	\label{lemma:tau_d}
	Let $ g $ be the grade of $ \bb $ with respect to $ \AA $  as in \cref{def:grade_b}. We have 
	\begin{align*}
		\dd_{k} = \sum_{l=0}^{k-1} \frac{(-1)^{l} q_{(k,l)} }{\prod_{\ell=0}^{l} \gamma_{k-l+\ell}^{(2)}} \vv_{k-l}, \quad 1 \leq k < g,
	\end{align*}
	where $q_{(k,l)} $ is defined in \cref{eq:det_T_R:S} and $q_{(k,0)} = 1 $.
	Furthermore, if $ \gamma^{(2)}_{g} \neq 0 $, the conclusion continues to hold for $ k = g $.
\end{lemma}
\begin{proof}
	The proof follows by expanding the definition of $ \dd_{k} $ from the construction of \cref{alg:MINRES} as well as using using the three-term recurrence relation \cref{eq:det_iters}.
	Indeed, we have
	\begin{align*}
		\dd_{k} &= \frac{1}{\gamma_{k}^{(2)}} \left( \vv_{k} - \delta_{k}^{(2)} \dd_{k-1} - \epsilon_{k} \dd_{k-2}\right) \\
		&= \frac{1}{\gamma_{k}^{(2)}}  \vv_{k} - \frac{\delta_{k}^{(2)}}{\gamma_{k-1}^{(2)} \gamma_{k}^{(2)}}   \left(\vv_{k-1} - \delta_{k-1}^{(2)} \dd_{k-2} - \epsilon_{k-1} \dd_{k-3}\right) - \frac{\epsilon_{k}}{\gamma_{k}^{(2)}}  \dd_{k-2} \\
		&= \frac{1}{\gamma_{k}^{(2)}}  \vv_{k} - \frac{\delta_{k}^{(2)}}{\gamma_{k-1}^{(2)} \gamma_{k}^{(2)}}   \vv_{k-1} + \frac{\delta_{k-1}^{(2)} \delta_{k}^{(2)} - \gamma_{k-1}^{(2)} \epsilon_{k}}{\gamma_{k-1}^{(2)} \gamma_{k}^{(2)}}  \dd_{k-2} + \frac{\epsilon_{k-1} \delta_{k}^{(2)}}{\gamma_{k-1}^{(2)} \gamma_{k}^{(2)}}  \dd_{k-3} \\
		&= \frac{1}{\gamma_{k}^{(2)}}  \vv_{k} - \frac{\delta_{k}^{(2)}}{\gamma_{k-1}^{(2)} \gamma_{k}^{(2)}}   \vv_{k-1}  + \frac{\delta_{k-1}^{(2)} \delta_{k}^{(2)} - \gamma_{k-1}^{(2)} \epsilon_{k}}{\gamma_{k-2}^{(2)} \gamma_{k-1}^{(2)} \gamma_{k}^{(2)}}  \left(\vv_{k-2} - \delta_{k-2}^{(2)} \dd_{k-3} - \epsilon_{k-2} \dd_{k-4}\right) + \frac{\epsilon_{k-1} \delta_{k}^{(2)}}{\gamma_{k-1}^{(2)} \gamma_{k}^{(2)}}  \dd_{k-3} \\
		&= \frac{1}{\gamma_{k}^{(2)}}  \vv_{k} - \frac{\delta_{k}^{(2)}}{\gamma_{k-1}^{(2)} \gamma_{k}^{(2)}}   \vv_{k-1} + \frac{\delta_{k-1}^{(2)} \delta_{k}^{(2)} - \gamma_{k-1}^{(2)} \epsilon_{k}}{\gamma_{k-2}^{(2)} \gamma_{k-1}^{(2)} \gamma_{k}^{(2)}}  \vv_{k-2} \\
		&\quad - \frac{\delta_{k-2}^{(2)}\left(\delta_{k-1}^{(2)} \delta_{k}^{(2)} - \gamma_{k-1}^{(2)} \epsilon_{k}\right) - \gamma_{k-2}^{(2)} \epsilon_{k-1} \delta_{k}^{(2)}}{\gamma_{k-2}^{(2)} \gamma_{k-1}^{(2)} \gamma_{k}^{(2)}}  \dd_{k-3} - \frac{\epsilon_{k-2}\left(\delta_{k-1}^{(2)} \delta_{k}^{(2)} - \gamma_{k-1}^{(2)} \epsilon_{k}\right)}{\gamma_{k-2}^{(2)} \gamma_{k-1}^{(2)} \gamma_{k}^{(2)}}  \dd_{k-4} \\
		&= \frac{q_{(k,0)}}{\gamma_{k}^{(2)}}  \vv_{k} - \frac{q_{(k,1)}}{\gamma_{k-1}^{(2)} \gamma_{k}^{(2)}}   \vv_{k-1} + \frac{q_{(k,2)}}{\gamma_{k-2}^{(2)} \gamma_{k-1}^{(2)} \gamma_{k}^{(2)}}  \vv_{k-2}  - \frac{\delta_{k-2}^{(2)}q_{(k,2)} - \gamma_{k-2}^{(2)} \epsilon_{k-1} q_{(k,1)}}{\gamma_{k-2}^{(2)} \gamma_{k-1}^{(2)} \gamma_{k}^{(2)}}  \dd_{k-3} \\
		& \quad - \frac{\epsilon_{k-2} q_{(k,2)}}{\gamma_{k-2}^{(2)} \gamma_{k-1}^{(2)} \gamma_{k}^{(2)}}  \dd_{k-4} \\
		&= \frac{q_{(k,0)}}{\gamma_{k}^{(2)}}  \vv_{k} - \frac{q_{(k,1)}}{\gamma_{k-1}^{(2)} \gamma_{k}^{(2)}}   \vv_{k-1} + \frac{q_{(k,2)}}{\gamma_{k-2}^{(2)} \gamma_{k-1}^{(2)} \gamma_{k}^{(2)}}  \vv_{k-2} - \frac{q_{(k,3)}}{\gamma_{k-2}^{(2)} \gamma_{k-1}^{(2)} \gamma_{k}^{(2)}}  \dd_{k-3} - \frac{\epsilon_{k-2} q_{(k,2)}}{\gamma_{k-2}^{(2)} \gamma_{k-1}^{(2)} \gamma_{k}^{(2)}}  \dd_{k-4}. 
	\end{align*}
	Continuing to expand in this way, we obtain
	\begin{align*}
		\dd_{k} &= \frac{q_{(k,0)}}{\gamma_{k}^{(2)}}  \vv_{k} - \frac{q_{(k,1)}}{\gamma_{k-1}^{(2)} \gamma_{k}^{(2)}}   \vv_{k-1} + \frac{q_{(k,2)}}{\gamma_{k-2}^{(2)} \gamma_{k-1}^{(2)} \gamma_{k}^{(2)}}  \vv_{k-2} - \frac{q_{(k,3)}}{\gamma_{k-3}^{(2)} \gamma_{k-2}^{(2)} \gamma_{k-1}^{(2)} \gamma_{k}^{(2)}}  \vv_{k-3} + \ldots \\
		&\quad +\frac{\delta_{k-l}^{(2)} q_{(k,l)} - \gamma_{k-l-1}^{(2)} \epsilon_{k-l} q_{(k,l-1)}}{\gamma_{k-l}^{(2)} \ldots \gamma_{k-1}^{(2)} \gamma_{k}^{(2)}} (-1)^{l+1}  \dd_{k-l-1} + \frac{\epsilon_{k-l} q_{(k,l)}}{\gamma_{k-l}^{(2)} \ldots \gamma_{k-1}^{(2)} \gamma_{k}^{(2)}} (-1)^{l+2}  \dd_{k-l-2} \\
		&\quad \vdots \\
		&= \sum_{l=0}^{k-1} \frac{(-1)^{l} q_{(k,l)}}{\prod_{\ell=0}^{l} \gamma_{k-l+\ell}^{(2)}}  \vv_{k-l}.
	\end{align*}
\end{proof}

\begin{proof}[Proof of \cref{thm:determinants}]
	Recall the three-term recurrence relation \cref{eq:det_iters} and the ansatz \cref{eq:det_R_formula}. We proceed to prove \cref{eq:det_R_formula} by induction. 
	By \cref{alg:MINRES}, it is easy to verify that $ q_{(k,1)} $ satisfies \cref{eq:det_R_formula}. Indeed, we have 
	\begin{align*}
		q_{(k,1)} &= \delta_{k}^{(2)} = c_{k-1} \delta_{k}^{(1)} + s_{k-1} \alpha_{k} = -c_{k-1} c_{k-2} \beta_{k} + s_{k-1} \alpha_{k} \\
		&= -\frac{\gamma_{k-1}^{(1)}}{\gamma_{k-1}^{(2)}} c_{k-2} \beta_{k} + \frac{\beta_{k}}{\gamma_{k-1}^{(2)}} \alpha_{k} = \left( \frac{p_{(k,1)}}{\gamma_{k-1}^{(2)}} - c_{k-2} \frac{\gamma_{k-1}^{(1)} p_{(k,0)}}{\gamma_{k-1}^{(2)}} \right) \beta_{k}.
	\end{align*}
	By constructions in \cref{alg:MINRES} as well as \cref{eq:c_s_r2}, we have
	\begin{align*}
		\delta_{k}^{(2)} \delta_{k-1}^{(2)} &= (c_{k-1} \delta_{k}^{(1)} + s_{k-1} \alpha_{k}) (c_{k-2} \delta_{k-1}^{(1)} + s_{k-2} \alpha_{k-1}) \\
		&= (-c_{k-1} c_{k-2} \beta_{k} + s_{k-1} \alpha_{k}) (-c_{k-2} c_{k-3} \beta_{k-1} + s_{k-2} \alpha_{k-1}) \\
		&= \left( \frac{-c_{k-2} \gamma_{k-1}^{(1)} + \alpha_{k} }{\gamma_{k-1}^{(2)}} \right) \left(\frac{-c_{k-3} \gamma_{k-2}^{(1)} + \alpha_{k-1}}{\gamma_{k-2}^{(2)}}\right) \beta_{k} \beta_{k-1}, 
	\end{align*}
	and
	\begin{align*}
		\gamma_{k-1}^{(2)} \epsilon_{k} = \left(\frac{(\gamma_{k-1}^{(1)})^2 + \beta_{k}^2}{\gamma_{k-1}^{(2)}}\right) \left(\frac{\beta_{k-1}}{\gamma_{k-2}^{(2)}}\right) \beta_{k}. 
	\end{align*}
	Putting this all together, it follows that
	\begin{align*}
		q_{(k,2)} &= \delta_{k}^{(2)} \delta_{k-1}^{(2)} - \gamma_{k-1}^{(2)} \epsilon_{k} \\
		&= \frac{\beta_{k-1} \beta_{k}}{\gamma_{k-2}^{(2)}\gamma_{k-1}^{(2)}} \left( (-c_{k-2} \gamma_{k-1}^{(1)} + \alpha_{k})(-c_{k-3} \gamma_{k-2}^{(1)} + \alpha_{k-1}) - (\gamma_{k-1}^{(1)})^2 - \beta_{k}^2 \right) \\
		&= \frac{\beta_{k-1} \beta_{k}}{\gamma_{k-2}^{(2)}\gamma_{k-1}^{(2)}} \left(\alpha_{k-1} \alpha_{k} - \beta_{k}^{2} - c_{k-3} \gamma_{k-2}^{(1)} \alpha_{k} + \gamma_{k-1}^{(1)} (c_{k-3} c_{k-2} \gamma_{k-2}^{(1)} - c_{k-2} \alpha_{k-1} - \gamma_{k-1}^{(1)}) \right) \\
		&= \frac{\beta_{k-1} \beta_{k}}{\gamma_{k-2}^{(2)}\gamma_{k-1}^{(2)}} \left(\alpha_{k-1} \alpha_{k} - \beta_{k}^{2} - c_{k-3} \gamma_{k-2}^{(1)} \alpha_{k} + c_{k-3} \gamma_{k-1}^{(1)} \gamma_{k-2}^{(2)} \right),
	\end{align*}
	where the last equality follows since 
	\begin{align*}
		\gamma_{k-1}^{(1)} \left( c_{k-3} c_{k-2} \gamma_{k-2}^{(1)} - c_{k-2} \alpha_{k-1} - \gamma_{k-1}^{(1)} \right) = \gamma_{k-1}^{(1)} \left( c_{k-3} c_{k-2} \gamma_{k-2}^{(1)} + c_{k-3} s_{k-2} \beta_{k-1} \right) = c_{k-3} \gamma_{k-1}^{(1)} \gamma_{k-2}^{(2)}.
	\end{align*}
	Hence, it follows that $ q_{(k,2)} $ also satisfies \cref{eq:det_R_formula} since
	\begin{align*}
		q_{(k,2)} &= \frac{\beta_{k-1} \beta_{k}}{\gamma_{k-2}^{(2)}\gamma_{k-1}^{(2)}} \left(\alpha_{k-1} \alpha_{k} - \beta_{k}^{2} - c_{k-3} \gamma_{k-2}^{(1)} \alpha_{k} + c_{k-3} \gamma_{k-1}^{(1)} \gamma_{k-2}^{(2)} \right) \\
		&= \frac{\beta_{k-1} \beta_{k}}{\gamma_{k-2}^{(2)}\gamma_{k-1}^{(2)}} \left( p_{(k,2)} - c_{k-3} \gamma_{k-2}^{(1)} p_{(k,1)} + c_{k-3} \gamma_{k-1}^{(1)} \gamma_{k-2}^{(2)} p_{(k,0)} \right) \\
		&= \left( \frac{p_{(k,2)}}{\prod_{i=1}^{2} \gamma_{k-i}^{(2)}} + c_{k-3} \sum_{i=1}^{2} \frac{(-1)^{i-1} \gamma_{k-i}^{(1)} p_{(k,i-1)}}{\prod_{j=1}^{i} \gamma_{k-j}^{(2)}} \right) \prod_{i=1}^{2} \beta_{k-i+1}.
	\end{align*}
	Now, suppose \cref{eq:det_R_formula} holds for $ q_{(k,l-1)}$ and $q_{(k,l-2)}$ for any $ 3 \leq l \leq k-1 $. By \cref{alg:MINRES,eq:c_s_r2}, we have
	\begin{align}
		\delta_{k-l+1}^{(2)} q_{(k,l-1)}  &=  \left( \frac{\alpha_{k-l+1} \beta_{k-l+1}  - c_{k-l-1} \gamma_{k-l}^{(1)} \beta_{k-l+1} }{\gamma_{k-l}^{(2)}} \right) \nonumber \\
		& \quad \times \left( \frac{p_{(k,l-1)}}{\prod_{i=1}^{l-1} \gamma_{k-i}^{(2)}} + c_{k-l} \sum_{i=1}^{l-1} \frac{(-1)^{l-i} \gamma_{k-i}^{(1)} p_{(k,i-1)}}{\prod_{j=1}^{i} \gamma_{k-j}^{(2)}} \right) \prod_{i=1}^{l-1} \beta_{k-i+1} \nonumber \\
		&= \left( \frac{\alpha_{k-l+1} - c_{k-l-1} \gamma_{k-l}^{(1)}}{\gamma_{k-l}^{(2)}} \right) \nonumber \\
		& \quad \times \left( \frac{p_{(k,l-1)}}{\prod_{i=1}^{l-1} \gamma_{k-i}^{(2)}} + c_{k-l} \sum_{i=1}^{l-1} \frac{(-1)^{l-i} \gamma_{k-i}^{(1)} p_{(k,i-1)}}{\prod_{j=1}^{i} \gamma_{k-j}^{(2)}} \right) \prod_{i=1}^{l} \beta_{k-i+1},
		\label{eq:delta_delta}
	\end{align}
	and
	\begin{align}
		\gamma_{k-l+1}^{(2)} \epsilon_{k-l+2} q_{(k,l-2)} &= \gamma_{k-l+1}^{(2)} \left( \frac{
			\beta_{k-l+1} \beta_{k-l+2}}{\gamma_{k-l}^{(2)}} \right) \nonumber \\
		& \quad \times \left( \frac{p_{(k,l-2)} }{\prod_{i=1}^{l-2} \gamma_{k-i}^{(2)}} + c_{k-l+1} \sum_{i=1}^{l-2} \frac{(-1)^{l-i-1} \gamma_{k-i}^{(1)} p_{(k,i-1)}}{\prod_{j=1}^{i} \gamma_{k-j}^{(2)}} \right) \prod_{i=1}^{l-2} \beta_{k-i+1} \nonumber \\
		&= \left( \frac{(\gamma_{k-l+1}^{(1)})^2 + \beta_{k-l+2}^2}{\gamma_{k-l+1}^{(2)} \gamma_{k-l}^{(2)}} \right) \nonumber \\
		& \quad \times \left( \frac{p_{(k,l-2)}}{\prod_{i=1}^{l-2} \gamma_{k-i}^{(2)}} + c_{k-l+1} \sum_{i=1}^{l-2} \frac{(-1)^{l-i-1} \gamma_{k-i}^{(1)} p_{(k,i-1)}}{\prod_{j=1}^{i} \gamma_{k-j}^{(2)}} \right) \prod_{i=1}^{l} \beta_{k-i+1}.
		\label{eq:gamma_epsilon}
	\end{align}
	
	We now verify that \cref{eq:det_R_formula} is indeed the difference of \cref{eq:delta_delta,eq:gamma_epsilon} and hence agrees with \cref{eq:det_iters}. Factoring out the common term ``$ \prod_{i=1}^{l} \beta_{k-i+1} $'', we now combine and simplify some terms in this difference. 
	
	We first note that 
	\begin{align}
		\label{eq:R_part1}  
		\frac{\alpha_{k-l+1}}{\gamma_{k-l}^{(2)}} \frac{p_{(k,l-1)}}{\prod_{i=1}^{l-1} \gamma_{k-i}^{(2)}} - \frac{\beta_{k-l+2}^2}{\gamma_{k-l+1}^{(2)} \gamma_{k-l}^{(2)}} \frac{p_{(k,l-2)}}{\prod_{i=1}^{l-2} \gamma_{k-i}^{(2)}} = \frac{\alpha_{k-l+1} p_{(k,l-1)} - \beta_{k-l+2}^2 p_{(k,l-2)}}{\prod_{i=1}^{l} \gamma_{k-i}^{(2)}} = \frac{p_{(k,l)}}{\prod_{i=1}^{l} \gamma_{k-i}^{(2)}},
	\end{align}
	and
	\begin{align}   
		\label{eq:R_part2}  
		\frac{- c_{k-l-1} \gamma_{k-l}^{(1)}}{\gamma_{k-l}^{(2)}} \frac{p_{(k,l-1)}}{\prod_{i=1}^{l-1} \gamma_{k-i}^{(2)}} = \frac{- c_{k-l-1} \gamma_{k-l}^{(1)} p_{(k,l-1)}}{\prod_{i=1}^{l} \gamma_{k-i}^{(2)}}.
	\end{align}
	%
	We have
	\begin{align*}
		&\quad - \left( \alpha_{k-l+1} - c_{k-l-1} \gamma_{k-l}^{(1)} \right) c_{k-l} \gamma_{k-l+1}^{(1)} - (\gamma_{k-l+1}^{(1)})^2 \\
		&= \gamma_{k-l+1}^{(1)} \left( -c_{k-l} \alpha_{k-l+1} + c_{k-l} c_{k-l-1} \gamma_{k-l}^{(1)} - \left( - c_{k-l-1} s_{k-l} \beta_{k-l+1} - c_{k-l} \alpha_{k-l+1} \right) \right) \\
		&= c_{k-l-1} \gamma_{k-l+1}^{(1)} \left( c_{k-l} \gamma_{k-l}^{(1)} + s_{k-l} \beta_{k-l+1} \right) = c_{k-l-1} \gamma_{k-l+1}^{(1)} \frac{(\gamma_{k-l}^{(1)})^2 + \beta_{k-l+1}^2}{\gamma_{k-l}^{(2)}} = c_{k-l-1} \gamma_{k-l+1}^{(1)} \gamma_{k-l}^{(2)}.
	\end{align*}
	This gives
	\begin{align}
		\label{eq:R_part3}  
		\left(\frac{\alpha_{k-l+1} - c_{k-l-1} \gamma_{k-l}^{(1)}}{\gamma_{k-l}^{(2)}}\right) \left(\frac{-c_{k-l} \gamma_{k-l+1}^{(1)} p_{(k,l-2)}}{\prod_{j=1}^{l-1} \gamma_{k-j}^{(2)}}\right) - \frac{(\gamma_{k-l+1}^{(1)})^2}{\gamma_{k-l+1}^{(2)} \gamma_{k-l}^{(2)}} \frac{p_{(k,l-2)}}{\prod_{i=1}^{l-2} \gamma_{k-i}^{(2)}} = \frac{c_{k-l-1} \gamma_{k-l+1}^{(1)} p_{(k,l-2)}}{\prod_{i=1}^{l-1} \gamma_{k-i}^{(2)}}.
	\end{align}
	Finally, for any $ 1 \leq i \leq l-2 $, we have 
	\begin{align}   
		& \quad \left(\frac{\alpha_{k-l+1} - c_{k-l-1} \gamma_{k-l}^{(1)}}{\gamma_{k-l}^{(2)}}\right) \left(\frac{c_{k-l} (-1)^{l-i} \gamma_{k-i}^{(1)} p_{(k,i-1)}}{\prod_{j=1}^{i} \gamma_{k-j}^{(2)}}\right) - \left( \frac{\gamma_{k-l+1}^{(2)}}{\gamma_{k-l}^{(2)}} \right) \frac{c_{k-l+1} (-1)^{l-i-1} \gamma_{k-i}^{(1)} p_{(k,i-1)}}{\prod_{j=1}^{i} \gamma_{k-j}^{(2)}} \nonumber \\
		&= \left( \frac{- c_{k-l} \alpha_{k-l+1} + c_{k-l} c_{k-l-1} \gamma_{k-l}^{(1)}}{\gamma_{k-l}^{(2)}} - \frac{c_{k-l+1} \gamma_{k-l+1}^{(2)}}{\gamma_{k-l}^{(2)}} \right) \frac{ (-1)^{l-i+1} \gamma_{k-i}^{(1)} p_{(k,i-1)}}{\prod_{j=1}^{i} \gamma_{k-j}^{(2)}} \nonumber \\
		&= \left( \frac{- c_{k-l} \alpha_{k-l+1} + c_{k-l} c_{k-l-1} \gamma_{k-l}^{(1)} - \gamma_{k-l+1}^{(1)}}{\gamma_{k-l}^{(2)}}\right) \frac{ (-1)^{l-i+1} \gamma_{k-i}^{(1)} p_{(k,i-1)}}{\prod_{j=1}^{i} \gamma_{k-j}^{(2)}} \nonumber \\
		&= \left( \frac{- c_{k-l} \alpha_{k-l+1} + c_{k-l} c_{k-l-1} \gamma_{k-l}^{(1)} - \left( - c_{k-l-1} s_{k-l} \beta_{k-l+1} - c_{k-l} \alpha_{k-l+1} \right)}{c_{k-l} \gamma_{k-l}^{(1)} + s_{k-l} \beta_{k-l+1}}\right) \frac{ (-1)^{l-i+1} \gamma_{k-i}^{(1)} p_{(k,i-1)}}{\prod_{j=1}^{i} \gamma_{k-j}^{(2)}} \nonumber \\
		&= \left( \frac{c_{k-l-1} (c_{k-l} \gamma_{k-l}^{(1)} + s_{k-l} \beta_{k-l+1})}{c_{k-l} \gamma_{k-l}^{(1)} + s_{k-l} \beta_{k-l+1}}\right) \frac{ (-1)^{l-i+1} \gamma_{k-i}^{(1)} p_{(k,i-1)}}{\prod_{j=1}^{i} \gamma_{k-j}^{(2)}} = c_{k-l-1} \frac{(-1)^{l-i+1} \gamma_{k-i}^{(1)} p_{(k,i-1)}}{\prod_{j=1}^{i} \gamma_{k-j}^{(2)}}. \label{eq:R_part4}
	\end{align} 
	
	Now, summing up \cref{eq:R_part1,eq:R_part2,eq:R_part3,eq:R_part4}, and putting ``$ \prod_{i=1}^{l} \beta_{k-i+1} $'' factor back in, we obtain \cref{eq:det_R_formula}. 
	The claim follows by noticing that from \cref{thm:determinant_T}, \cref{eq:R_entries,eq:c_gama_tau_c}, all of \cref{eq:R_part1,eq:R_part2,eq:R_part3,eq:R_part4} evaluate to be positive. Furthermore, by \cref{eq:alpha_beta}, we also have $ \prod_{i=1}^{l} \beta_{k-i+1} > 0 $. Putting this all together, we get the desired result.
\end{proof}

\cref{lemma:tau_dTr} allow us to make a statement about the ``angle'' between the update vector $ \tau_{k}\dd_{k} $ and the residual vectors, which is used in the proofs of \cref{thm:A_norm,thm:xTr_monotonicity}.

\begin{lemma}
	\label{lemma:tau_dTr}
	Let $ g $ be the grade of $ \bb $ with respect to $ \AA $ as in \cref{def:grade_b}. As long as the NPC condition \cref{eq:c_gama_nc} has not been detected for $1 \leq k < g $, we have
	\begin{align*}
		\tau_{k} \dd_{k}^{\T} \rr_{k-j} > 0, \quad 0 \leq j < k.
	\end{align*}
	Furthermore, if $ \gamma^{(2)}_{g} \neq 0 $ and the NPC condition \cref{eq:c_gama_nc} has not been detected for all $ g $ iterations, the conclusion continues to hold for $ k = g $ and $ 1 \leq j < g $.
\end{lemma}
\begin{proof}
	This easily follows from \cref{lemma:tvr,lemma:tau_d},\cref{thm:determinants}, and \cref{eq:R_entries}, because
	\begin{align*}
		\tau_{k} \dd_{k}^{\T} \rr_{k-j} = \tau_{k} \sum_{i=0}^{k-1} \frac{(-1)^{i} q_{(k,i)} }{\prod_{\ell=0}^{i} \gamma_{k-i+\ell}^{(2)}} \vv_{k-i}^{\T} \rr_{k-j} > 0,
	\end{align*}
	where we also used the fact that $ \vv_{k-i} \perp \rr_{k-j}, \; i < j - 1 $.
\end{proof}

The following technical lemmas (\cref{lemma:tvtd,lemma:tau_dTx,lemma:tau_dTb}) will be used in the proof of \cref{thm:A_norm}. 
\begin{lemma}
	\label{lemma:tvtd}
	Let $ g $ be the grade of $ \bb $ with respect to $ \AA $  as in \cref{def:grade_b}. As long as the NPC condition \cref{eq:c_gama_nc} has not been detected for $1 \leq k < g $, we have
	\begin{align*}
		(-1)^{i} \tau_{k} \tau_{k-j} \dd_{k-j}^{\T} \vv_{k-i} > 0, \quad 0 \leq j \leq i < k.
	\end{align*}
	Furthermore, if $ \gamma^{(2)}_{g} \neq 0 $ and the NPC condition \cref{eq:c_gama_nc} has not been detected for all $ g $ iterations, the conclusion continues to hold for $ k = g $.
\end{lemma}
\begin{proof}
	By \cref{lemma:tau_d}, we have
	\begin{align*}
		\vv_{k-i}^{\T}\dd_{k-j} = \vv_{k-i}^{\T} \left( \sum_{h=0}^{k-j-1} \frac{(-1)^{h} q_{(k-j,h)} }{\prod_{\ell=0}^{h} \gamma_{k-j-h+\ell}^{(2)}} \vv_{k-j-h} \right) = \frac{(-1)^{i-j} q_{(k-j,i-j)} }{\prod_{\ell=0}^{i-j} \gamma_{k-i+\ell}^{(2)}}.
	\end{align*}
	Also, by \cref{eq:c_gama_tau_c} and the definition of $ \tau_{k-j} $ in \cref{alg:MINRES}, we have
	\begin{align*}
		(-1)^{-j} \tau_{k} \tau_{k-j} = \phi_{k-j-1} (-1)^{j} c_{k-j} \tau_{k} > 0.
	\end{align*}
	Hence, from \cref{thm:determinants,eq:R_entries}, it follows that
	\begin{align*}
		(-1)^{i} \tau_{k} \tau_{k-j} \dd_{k-j}^{\T} \vv_{k-i} = (-1)^{i} \tau_{k} \frac{q_{(k-j,i-j)}}{\prod_{\ell=0}^{i-j} \gamma_{k-i+\ell}^{(2)}} (-1)^{i-j} \tau_{k-j} = (-1)^{-j} \tau_{k} \tau_{k-j} \frac{q_{(k-j,i-j)}}{\prod_{\ell=0}^{i-j} \gamma_{k-i+\ell}^{(2)}} > 0.
	\end{align*}
\end{proof}

\begin{lemma}
	\label{lemma:tau_dTx}
	Let $ g $ be the grade of $ \bb $ with respect to $ \AA $  as in \cref{def:grade_b}. As long as the NPC condition \cref{eq:c_gama_nc} has not been detected for $1 \leq k < g $, we have
	\begin{align*}
		\tau_{k} \dd_{k}^{\T} \xx_{k-j} > 0, \quad 0 \leq j < k.
	\end{align*}
	Furthermore, if $ \gamma^{(2)}_{g} \neq 0 $ and the NPC condition \cref{eq:c_gama_nc} has not been detected for all $ g $ iterations, the conclusion continues to hold for $ k = g $.
\end{lemma}
\begin{proof}
	By \cref{lemma:tau_d} and noting that $ \xx_{k} \in \mathcal{K}_{k}(\AA, \bb) \perp \vv_{k+1}$, we have
	\begin{align*}
		\tau_{k} \dd_{k}^{\T} \xx_{k-j} &= \sum_{i=0}^{k-1} \frac{(-1)^{i} \tau_{k} q_{(k,i)}}{\prod_{\ell=0}^{i} \gamma_{k-i+\ell}^{(2)}}  \vv_{k-i}^{\T} \xx_{k-j} = \sum_{i=j}^{k-1} \frac{(-1)^{i} \tau_{k} q_{(k,i)}}{\prod_{\ell=0}^{i} \gamma_{k-i+\ell}^{(2)}} \vv_{k-i}^{\T} \xx_{k-j} \\
		&= \sum_{i=j}^{k-1} \frac{(-1)^{i} \tau_{k} q_{(k,i)}}{\prod_{\ell=0}^{i} \gamma_{k-i+\ell}^{(2)}} \vv_{k-i}^{\T} \left( \sum_{h=1}^{k-j} \tau_{h} \dd_{h} \right) = \sum_{i=j}^{k-1} \frac{(-1)^{i} \tau_{k} q_{(k,i)}}{\prod_{\ell=0}^{i} \gamma_{k-i+\ell}^{(2)}}  \left( \sum_{h=1}^{k-j} \tau_{h} \vv_{k-i}^{\T} \dd_{h} \right).
	\end{align*}
	Again from \cref{lemma:tau_d}, we have
	\begin{align*}
		\vv_{k-i}^{\T}\dd_{h} = \vv_{k-i}^{\T} \left( \sum_{\ell=0}^{h-1} \frac{(-1)^{\ell} q_{(h,\ell)} }{\prod_{\ell=0}^{\ell} \gamma_{h-\ell+\ell}^{(2)}} \vv_{h-\ell} \right) = 0, \quad h < k-i,
	\end{align*}
	which implies
	\begin{align*}
		\tau_{k} \dd_{k}^{\T} \xx_{k-j} = \sum_{i=j}^{k-1} \frac{(-1)^{i} \tau_{k} q_{(k,i)}}{\prod_{\ell=0}^{i} \gamma_{k-i+\ell}^{(2)}}  \left( \sum_{h=k-i}^{k-j} \tau_{h} \vv_{k-i}^{\T} \dd_{h} \right) = \sum_{i=j}^{k-1} \sum_{h=k-i}^{k-j} \frac{q_{(k,i)}}{\prod_{\ell=0}^{i} \gamma_{k-i+\ell}^{(2)}} \left( (-1)^{i} \tau_{k} \tau_{h} \dd_{h}^{\T} \vv_{k-i} \right)  > 0,
	\end{align*}
	where the last inequality follows from \cref{thm:determinants}, \cref{eq:R_entries}, and \cref{lemma:tvtd}.
\end{proof}

\begin{lemma}
	\label{lemma:tau_dTb}
	Let $ g $ be the grade of $ \bb $ with respect to $ \AA $ as in \cref{def:grade_b}. As long as the NPC condition \cref{eq:c_gama_nc} has not been detected for $1 \leq k < g $, we have
	\begin{align*}
		\tau_{k} \dd_{k}^{\T} \bb > 0.
	\end{align*}
	Furthermore, if $ \gamma^{(2)}_{g} \neq 0 $ and the NPC condition \cref{eq:c_gama_nc} has not been detected for all $ g $ iterations, the conclusion continues to hold for $ k = g $.
\end{lemma}
\begin{proof}
	By \cref{lemma:tau_d}, we have
	\begin{align*}
		\tau_{k} \dd_{k}^{\T} \bb &= \beta_{1} \left(\sum_{i=0}^{k-1} \frac{q_{(k,i)}}{\prod_{\ell=0}^{i} \gamma_{k-i+\ell}^{(2)}} (-1)^{i} \tau_{k} \vv_{k-i} \right)^{\T} \vv_{1} \\
		&= \frac{q_{(k,k-1)}}{\prod_{\ell=0}^{k-1} \gamma_{\ell+1}^{(2)}} (-1)^{k-1} \tau_{k} \vv_{1}^{\T} \vv_{1} = \frac{\beta_{1} q_{(k,k-1)}}{\prod_{\ell=0}^{k-1} \gamma_{\ell+1}^{(2)}} (-1)^{k} c_{0} \tau_{k} > 0,
	\end{align*}
	where the last inequality comes from \cref{eq:c_gama_tau_c}.
\end{proof}

	\bibliographystyle{plain}
	\bibliography{biblio}

\begin{thebibliography}{10}

\bibitem{bjorck2015numerical}
Ake Bj{\"o}rck.
\newblock {\em {Numerical Methods in Matrix Computations}}, volume~59.
\newblock Springer International Publishing, 2015.

\bibitem{choi2011minres}
Sou-Cheng~T Choi, Christopher~C Paige, and Michael~A Saunders.
\newblock {MINRES-QLP: A Krylov subspace method for indefinite or singular
  symmetric systems}.
\newblock {\em SIAM Journal on Scientific Computing}, 33(4):1810--1836, 2011.

\bibitem{choi2014algorithm}
Sou-Cheng~T Choi and Michael~A Saunders.
\newblock {Algorithm 937: MINRES-QLP for symmetric and Hermitian linear
  equations and least-squares problems}.
\newblock {\em ACM Transactions on Mathematical Software (TOMS)}, 40(2):16,
  2014.

\bibitem{conn2000trust}
Andrew~R Conn, N.~I.~M. Gould, and {\relax Ph}ilip~L. Toint.
\newblock {\em {Trust Region Methods}}, volume~1.
\newblock SIAM, 2000.

\bibitem{curtis2021trust}
Frank~E Curtis, Daniel~P Robinson, Cl{\'e}ment~W Royer, and Stephen~J Wright.
\newblock {Trust-region Newton-CG with strong second-order complexity
  guarantees for nonconvex optimization}.
\newblock {\em SIAM Journal on Optimization}, 31(1):518--544, 2021.

\bibitem{dahito2019conjugate}
Marie-Ange Dahito and Dominique Orban.
\newblock {The Conjugate Residual Method in Linesearch and Trust-Region
  Methods}.
\newblock {\em SIAM Journal on Optimization}, 29(3):1988--2025, 2019.

\bibitem{demmel1997applied}
James~W Demmel.
\newblock {\em {Applied numerical linear algebra}}.
\newblock SIAM, Philadelphia, 1997.

\bibitem{el2004inverse}
Moawwad~EA El-Mikkawy.
\newblock {On the inverse of a general tridiagonal matrix}.
\newblock {\em Applied Mathematics and Computation}, 150(3):669--679, 2004.

\bibitem{fong2012cg}
David Chin-Lung Fong and Michael Saunders.
\newblock {CG versus MINRES: An empirical comparison}.
\newblock {\em Sultan Qaboos University Journal for Science [SQUJS]},
  17(1):44--62, 2012.

\bibitem{freund1994new}
Roland~W Freund and No{\"e}l~M Nachtigal.
\newblock {A new Krylov-subspace method for symmetric indefinite linear
  systems}.
\newblock Technical report, Oak Ridge National Lab., TN (United States), 1994.

\bibitem{golub2012matrix}
Gene~H. Golub and Charles~F. Van~Loan.
\newblock {\em {Matrix Computations}}.
\newblock Johns Hopkins Studies in the Mathematical Sciences. Johns Hopkins
  University Press, 4 edition, 2012.

\bibitem{hnvetynkova2007lanczos}
Iveta Hn{\v{e}}tynkov{\'a} and Zden{\v{e}}k Strako{\v{s}}.
\newblock {Lanczos tridiagonalization and core problems}.
\newblock {\em Linear Algebra and its Applications}, 421(2-3):243--251, 2007.

\bibitem{horn2012matrix}
Roger~A Horn and Charles~R Johnson.
\newblock {\em {Matrix analysis}}.
\newblock Cambridge university press, 2012.

\bibitem{krizhevsky2009learning}
Alex Krizhevsky and Geoffrey Hinton.
\newblock {Learning multiple layers of features from tiny images}.
\newblock 2009.
\newblock Data available at \url{https://www.cs.toronto.edu/~kriz/cifar.html}.

\bibitem{liu2021convergence}
Yang Liu and Fred Roosta.
\newblock {Convergence of Newton-MR under inexact Hessian information}.
\newblock {\em SIAM Journal on Optimization}, 31(1):59--90, 2021.

\bibitem{mishra2008invexity}
Shashi~K Mishra and Giorgio Giorgi.
\newblock {\em {Invexity and Optimization}}, volume~88.
\newblock Springer Science \& Business Media, 2008.

\bibitem{nocedal2006numerical}
Jorge Nocedal and Stephen Wright.
\newblock {\em {Numerical Optimization}}.
\newblock Springer Science \& Business Media, 2006.

\bibitem{o2021log}
Michael O’Neill and Stephen~J Wright.
\newblock {A log-barrier Newton-CG method for bound constrained optimization
  with complexity guarantees}.
\newblock {\em IMA Journal of Numerical Analysis}, 41(1):84--121, 2021.

\bibitem{paige1975solution}
Christopher~C Paige and Michael~A Saunders.
\newblock {Solution of sparse indefinite systems of linear equations}.
\newblock {\em SIAM Journal on Numerical Analysis}, 12(4):617--629, 1975.

\bibitem{roosta2018newton}
Fred Roosta, Yang Liu, Peng Xu, and Michael~W Mahoney.
\newblock {Newton-MR: Inexact newton method with minimum residual sub-problem
  solver}.
\newblock {\em arXiv preprint arXiv:1810.00303}, 2018.

\bibitem{royer2020newton}
Cl{\'e}ment~W Royer, Michael O’Neill, and Stephen~J Wright.
\newblock {A Newton-CG algorithm with complexity guarantees for smooth
  unconstrained optimization}.
\newblock {\em Mathematical Programming}, 180(1):451--488, 2020.

\bibitem{saad2003iterative}
Yousef Saad.
\newblock {\em {Iterative Methods for Sparse Linear Systems}}, volume~82.
\newblock SIAM, 2003.

\bibitem{steihaug1983conjugate}
Trond Steihaug.
\newblock {The conjugate gradient method and trust regions in large scale
  optimization}.
\newblock {\em SIAM Journal on Numerical Analysis}, 20(3):626--637, 1983.

\bibitem{stiefel1955relaxationsmethoden}
Edvard Stiefel.
\newblock {Relaxationsmethoden bester strategie zur l{\"o}sung linearer
  gleichungssysteme}.
\newblock {\em Commentarii Mathematici Helvetici}, 29(1):157--179, 1955.

\bibitem{trefethen1997numerical}
Lloyd~N Trefethen and David Bau~III.
\newblock {\em {Numerical Linear Algebra}}, volume~50.
\newblock SIAM, Philadelphia, 1997.

\bibitem{xie2021complexity01}
Yue Xie and Stephen~J Wright.
\newblock {Complexity of projected Newton methods for bound-constrained
  optimization}.
\newblock {\em arXiv preprint arXiv:2103.15989}, 2021.

\bibitem{xie2021complexity02}
Yue Xie and Stephen~J Wright.
\newblock {Complexity of proximal augmented Lagrangian for nonconvex
  optimization with nonlinear equality constraints}.
\newblock {\em Journal of Scientific Computing}, 86(3):1--30, 2021.

\bibitem{xuNonconvexEmpirical2017}
Peng Xu, Fred Roosta, and Michael~W Mahoney.
\newblock {Second-order optimization for non-convex machine learning: An
  empirical study}.
\newblock In {\em Proceedings of the 2020 SIAM International Conference on Data
  Mining}, pages 199--207. SIAM, 2020.

\end{thebibliography}

\end{document}